\documentclass[10pt]{amsart}

\usepackage{graphicx}
\usepackage{xcolor}
\usepackage{amsmath,amssymb,amsfonts,euscript,enumerate}
 \usepackage{amsthm}

\usepackage{color,wrapfig}
\usepackage{pxfonts}
\usepackage{verbatim}

\usepackage{multicol}
\usepackage{color}
\usepackage{xspace}

\usepackage[normalem]{ulem}

\newcommand{\bA}{\mathbf{A}}
\newcommand{\N}{\mathbb N}
\newcommand{\Z}{\mathbb Z}
\newcommand{\R}{{\mathbb R}}
\newcommand{\mS}{{\mathbb S}}

\newcommand{\cW}{{\mathcal W}}

\newcommand{\cI}{{\mathcal I}}

\newcommand{\cK}{{\mathcal K}}

\newcommand{\la}{\lambda}

\newcommand{\bnu}{\boldsymbol{\nu}}
\newcommand{\bxi}{\boldsymbol{\xi}}
\newcommand{\bla}{\boldsymbol{\lambda}}
\newcommand{\bX}{\mathbf{X}}

\newcommand{\hii}{h_{ii}}
\newcommand{\hjj}{h_{jj}}

\newcommand{\D}{\nabla}

\newcommand{\overDi}{\bar{\nabla_{i}}}
\newcommand{\overDj}{\bar{\nabla_{j}}}

\newcommand{\ra}{\rightarrow}

\newcommand{\pa}{\partial}
\newcommand{\on}{\overline{\nabla}}
\newcommand{\fs}{f_{\sigma}}
\newcommand{\na}{\nabla}
\newcommand{\tna}{\tilde{\na}}

\newcommand{\fr}{\frac}
\newcommand{\ep}{\epsilon}

\newcommand{\hk}{h_{kk}}
\newcommand{\hl}{h_{ll}}
\newcommand{\hi}{h_{ii}}
\newcommand{\hj}{h_{jj}}

\newcommand{\hone}{h_{11}}
\newcommand{\hn}{h_{nn}}

\newcommand{\fdi}{\dot{f}^{\,i}}

\newcommand{\fk}{\dot{f}^{\,k}}

\newcommand{\fone}{\dot{f}^{\,1}}

\newcommand{\fij}{\ddot{f}^{\,ij}}

\newcommand{\no}[1]{}

\newtheorem{theorem}{Theorem}[section]
\newtheorem{lemma}[theorem]{Lemma}

\newtheorem{definition}{Definition}[section]
\newtheorem{remark}[theorem]{Remark}
\newtheorem{corollary}[theorem]{Corollary}
\numberwithin{equation}{section}

\begin{document}

%\begin{frontmatter}

\title[Anisotropic flow of convex hypersurfaces]{Anisotropic flow of convex hypersurfaces by the square root of the scalar curvature}

\author{Hyunsuk Kang}

\author{Lami Kim}

\author{Ki-Ahm Lee}

\maketitle

\begin{abstract}
We show the existence of a smooth solution for the flow deformed by the square root of the scalar curvature multiplied by a positive anisotropic factor $\psi$ given a strictly convex initial hypersurface in Euclidean space suitably pinched. We also prove the convergence of rescaled surfaces to a smooth limit manifold 
which is a round sphere. In dimension two, it is shown that, with a volume preserving rescaling, the limit profile satisfies a soliton equation.
\end{abstract}

%\end{frontmatter}

\section{Introduction}

The evolution of hypersurfaces in Euclidean spaces governed by curvature has been considered in many aspects in geometric analysis and mathematical physics.  In this paper, we consider a one parameter family of immersions $X(\cdot, t):\mS^n \rightarrow M_t \subset \R^{n+1}$, $M_t:=X(\mS^n,t)$,
and its evolution in time governed by the square root of scalar curvature on $M_t$ and a given smooth positive function $\psi$ in $\R^{n+1}$ with compact support: consider the following initial value problem
  \begin{align}\label{evolution_eq1}
    \frac{\partial X}{\partial t} &= - \psi(X(x,t)) R(x,t)^{1/2} \bnu, \qquad
    M_0 = X(\mS^n, 0) ,
  \end{align}
where $\bnu$ is the outward unit normal to $M_t$, and $M_0$ is a strictly convex smooth hypersurface in $\R^{n+1}$.  Here $\psi$, by which we call the anisotropic factor, can be considered as a nonhomogeneous influence on the curvature flow from the underlying manifold $\R^{n+1}$.  Note the dependence of $\psi$ on the position $X(\cdot, t)$, not on the normal vector $\boldsymbol{\nu}$ as considered in \cite{A2,A7}. Throughout the paper, we shall call the flow in \eqref{evolution_eq1} {\it anisotropic scalar curvature flow} in short.  The flow we concern in this paper is the generalisation of that considered by Chow in \cite{C2} where the speed of the flow is $R^{1/2}$, and recently the case in which the speed is $R^p$ for $p > 1/2$, was studied in \cite{AS}.     Our aim is to show the smooth convergence of the flow and find a condition on $\psi$ to have a spherical limit profile of the rescaled flow.

\subsection{Notation}

In a local coordinates system $\{x_{1},\cdots,x_{n}\}$, the induced metric and the second fundamental form are given by
 \begin{align*}
g_{ij}=\left<\frac{\partial X}{\partial x^i},\frac{\partial X}{\partial x^i}\right>
\quad \text{and } \quad  h_{ij}=-\left<  \frac{\partial^2 X}{\partial x^i\partial x^j},\boldsymbol{\nu}\right> ,
  \end{align*}
respectively, where $\left< \cdot, \cdot \right>$ is the standard inner product and $\bnu$ is the outward unit normal vector to $M$.  In terms of these, the Weingarten map $\cW$ is given by
$$
\cW=(h^{i}_{j})=(g^{ik}h_{kj}),
$$
with the eigenvalues $\lambda_{1},\cdots,\lambda_{n}$, and its inverse given by $\cW^{-1}=(h^{-1})^i_j=\overline{g}^{ik}(h)_{kj}$, where $\overline{g}$ is the standard round metric on the n-dimensional sphere $\mS^n$ and $\overline{\nabla}$ is the connection of $\overline{g}$ on $\mS^n$.
Let
$\sigma_{k}=\sum_{1\leq i_{1} < \cdots < i_{k}\leq n}\lambda_{i_{1}}\lambda_{i_{2}}\cdots\la_{i_{k}}$ be the $k$-th symmetric function of the curvature , and one can write the mean curvature $H=\text{trace}(h^i_j)=\sigma_{1}=\sum_{1\leq i\leq n}\lambda_i$, the Gauss curvature
$K=\det (h)=\sigma_{n}=\lambda_{1}\lambda_{2}\cdots\la_{n}$ and the scalar curvature $R=\sum_{ i_{1} \neq  i_{2}}\lambda_{i_{1}}\lambda_{i_{2}}$ on which we shall focus in this paper.

For a symmetric function $f$ on $\R^{n+1}$, denote by $\Gamma_2$ the connected component of the set
$\{\bla \in \R^n: f(\lambda)>0\}$ containing the positive cone $\Gamma_+$.
It follows from \cite{CNS} that $\Gamma_2$ is a cone with the property that for all $\bla \in\Gamma_2$,
\begin{equation*}
\begin{split}
\frac{\partial}{\partial\lambda_i}f(\lambda)^{1/2}>0,
\qquad
\frac{\partial^2}{\partial\lambda_i\partial\lambda_j}f(\lambda)^{1/2}\leq 0.
\end{split}
\end{equation*}
We consider the following parabolic flow which is expected to converge to a smooth hypersurface $M^*$:
  \begin{equation}\label{eq-main}
    X_t=-F(\cW,\boldsymbol{\nu})\boldsymbol{\nu}
  \end{equation}
where we take
$$
F(\cW,\boldsymbol{\nu})=\psi R^{1/2}(\boldsymbol{\nu},t) ,
$$
for the flow \eqref{evolution_eq1}.
This can be written in terms of the support function $S=S(z,t)=\langle z, X(t)\rangle$, $z \in \mS^n$, as
\begin{equation*}
S_t=- \psi \left(\frac{s_n}{ s_{n-2}}\right)^{1/2}=\Phi(\cW^{-1},z)
\end{equation*}
for $ \Phi (\cW^{-1},z)=-F(\cW,\boldsymbol{\nu}) $ for $z=\boldsymbol{\nu} \in \mS^n$.
For the anisotropic factor $\psi$, we denote the sup norm of its differentials by
$||D\psi||=\sup_{\bold{w} \in S^n, x \in \R^{n+1}} |D\psi(\bold{w})| (x)$ and $|| D^2 \psi || =\sup_{\bold{v} , \bold{w} \in S^n , x \in \R^{n+1}} |D^2 \psi ( \bold{v}, \bold{w})| (x)$, where $D$ is the gradient in $\R^{n+1}$.

Throughout the paper, $C$ denotes a positive constant depending only on the dimension $n$ and other fixed constants, and we write $c(a_1, \cdots, a_k)$ for a positive constant depending only on its arguments $a_1, \cdots, a_k$.

\subsection{History}
The well known example of evolution of hypersurfaces by curvature is the mean curvature flow (see \cite{B,CGG,CM,E1,EH,Hu1,Hu2} among many) for which excellent lecture notes \cite{E2,Sm,Wa,Z} are available, and others are
Gauss curvature flow (see, for example,  \cite{A1,A2,C1,CEI,DH,DL1,DL2,F,KLR}), the flows evolving with the speed of powers of mean curvature (see \cite{CRS,Sc1,Sc2}) and the flows by homogeneous functions of the principal curvatures \cite{A4,AM}. Most notably for our interest, for the flow deformed by powers of the scalar curvature in \cite{AS,C2}, they proved the short time existence and the long time existence as well as the convergence to a point, and also the convergence of the rescaled flow to a round sphere.
  The difference between the flow in \cite{C2} and \eqref{evolution_eq1} is the presence of the anisotropic factor $\psi$ and the limit profile is expected to satisfy a non-trivial limit equation.
  With a further assumption that the perturbation from $\psi$ is relatively small compared with the initial data, we show that the flow under the parabolic rescaling converges to a round sphere.
  Note that the flow \eqref{evolution_eq1} is somewhat related to the logarithmic Gauss curvature flow considered in \cite{CW} to solve the Minkowski problem where the evolution equation is given by
  \begin{align}\label{logarithmic_GCF}
    \frac{\partial X}{\partial t} &= -\log \frac{K(\bnu)}{f(\bnu)}\bnu,
    \\
    M_0 &= X(\mS^n, 0), \nonumber
  \end{align}
 where $K(\bnu)$ is the Gauss curvature of $M_t$ and $f$ is a positive smooth function on $\mS^n$.  Along this flow, the smoothness and the convexity of the hypersurfaces are preserved, and given that the weighted center of mass is equal to zero and starting from suitably chosen initial data , the limit profile of \eqref{logarithmic_GCF} satisfies
  \begin{align*}
  \log \frac{K(\bnu)}{f(\bnu)} &= 0 \, ,
  \end{align*}
  which is equivalent to have $K=f$ so that the given Borel measure of $\mS^n$ coincides with the area measure of the convex hypersurface.
  Likewise, consider the immersions of convex hypersurfaces with the evolution according to
    \begin{align}\label{evolution_eq2}
    \frac{\partial X}{\partial t} &= -\left(\fr{F}{f(\bnu,X)}-1 \right) \bnu \, ,
  \end{align}
  where $F$ is a function depending on the curvature of the hypersurface and $f$ is a function given {\it a priori}.
  Although it is not shown in this paper that under \eqref{evolution_eq1},  the limit hypersurface has its scalar curvature equal to a given smooth function on $\mS^n$, one may expect that depending on $f$, the flow \eqref{evolution_eq2} contracts to a point, expand to an asymptotic sphere or converges to a convex hypersurface with its limit profile satisfying $F=f$ under some conditions.

\subsection{Main Theorems}
We state the main results for the flow (\ref{evolution_eq1}).
\begin{theorem}\label{thm-main-1}
Let $M_0=X(S^n,0)$ be a compact, connected and strictly convex smooth manifold in $\R^{n+1}$.  Suppose that $h_{ij} \ge \epsilon (H+c)g_{ij}$ initially for some $\epsilon >0$ and $c>0$ satisfying
 \begin{align}\label{c_condition1}
   c &\ge \max \Big\{ \fr{5 n^2 ||D \psi||}{\ep^2 \psi} \, , \, \fr{3n ||D^2 \psi||^{1/2}}{\ep  \psi^{1/2}} + \fr{2n^3 ||D\psi||}{\psi} , 
   \fr{10}{\ep^6} \Big(  \fr{||D \psi||}{\psi} +  \fr{||D^2 \psi||}{\psi} \Big)  \Big\},
 \end{align}
where $D$ is the gradient in $\R^{n+1}$.  Then there exist a maximal time $T>0$ and a unique smooth solution $\{M_t=X(S^n,t)\}$ satisfying \eqref{evolution_eq1} for $t\in [0,T)$, and $M_t$ converges to a point $x_0=M^*$ as $t$ approaches $T$.
\end{theorem}
\begin{remark}
\item
\begin{enumerate}[(i)]
\item  From the pinching assumption at $t=0$, $c$ is related to initial data by $H \ge c n \ep$.  The condition \eqref{c_condition1} can be regarded as the balance between strict convexity and the perturbation of $\psi$ from a constant map.
\item  If the smallest positive principal curvature is large compared with the perturbation of $\psi$, the initial hypersurface satisfies  \eqref{c_condition1}.  Then a pinching estimate follows and the convexity of the hypersurfaces preserved. Otherwise $\psi$ dominates and the convexity of the hypersurfaces may not be preserved.
\item If  the smallest positive principal curvature is small,  then the perturbation of $\psi$ is required to be small for  \eqref{c_condition1}  to be satisfied.
\end{enumerate}
\end{remark}
In order to observe the behaviour of the solution near the maximal time $T$, we rescale the solution and the time parameter by
      \begin{align}\label{X_parametrisation}
       \tilde{X}(x, \tau) = \frac{X(x,t)-X(x,T)}{\sqrt{2(T-t)}},  
       \qquad 
          \tau = -\frac{1}{2}\log \Big( \frac{T-t}{T} \Big).
      \end{align}
      Then the rescaled equation of \eqref{eq-main} is
       \begin{align}\label{eq-s}
        \frac{\partial \tilde{X}}{\partial \tau} &= -\tilde{F} \tilde{\nu}(x,\tau) + \tilde{X}, \qquad \tilde{F}= \tilde{\psi} \tilde{R}^{1/2},
       \end{align}
on $\mS^n \times [0, \infty)$.
For the rescaled hypersurfaces $\tilde{M}_{\tau}$, we obtain the convergence to a smooth manifold:
\begin{theorem}\label{thm-main-2}
Under the hypotheses in Theorem \ref{thm-main-1}, the hypersurfaces
$\tilde{M}_{\tau}$ under the parabolic rescaling converge 
in the $C^{\infty}$-topology
 to a smooth manifold $\tilde{M}^*$ as $\tau$ approaches infinity. In addition, if $\psi$ has a strict local minimum at $x_0$, then $\tilde{M}^*$ is a round sphere $S^n$.
In dimension two, with a volume preserving rescaling, the limit hypersurface $\hat{M}^*$ satisfies the equation $\hat{S}^* = C \hat{\psi} (\hat{R}^*)^{1/2} $
 for some constant $C>0$, where $\hat{S}^{*}$ and $\hat{R}^*$ are the support function and the scalar curvature of $\hat{M}^*$, respectively.
\end{theorem}
\begin{remark}
 The convergence in $C^{\infty}$-topology here means the uniform convergence of the derivatives of rescaled second fundamental form of any order.  
The exponential decay rate of the convergence to a round sphere as in  \cite{Hu1} shall be dealt in the sequel to this paper. 
\end{remark}

\subsection{Outline}
The paper is organized as follows.
In  Section \ref{sec-evo}, we find the evolution equations of tensors related to curvature.
In Section \ref{sec-pinching}, a pinching estimate for the second fundamental form is shown. In general, the pinching estimate derived from a maximum principle to tensors plays a crucial role to prove the convergence of convex hypersurfaces.  In \cite{A6}, the second derivative pinching estimates for a class of nonlinear parabolic equations were shown when the function describing the speed satisfies some structural criteria.  For the flow \eqref{eq-main}, $F$ is the function composed of the anisotropic factor $\psi$ and second derivatives of $X$ which belongs to the space of concave functions. In order to control the trouble terms that appear in the modification of the maximum principle in \cite{A6}, an additional perturbation term is required in the pinching.
In Section \ref{higherA}, we give upper bounds for the higher derivatives of the second fundamental form in terms of initial data and the derivatives of $\psi$.
 In Section \ref{sec-curvature}, applying the pinching estimate, we obtain the uniform upper bound of curvature before the blow up if $M_t$ is smooth.  
In Section \ref{global_boundedness}, we show a global $L^p$ estimate of a scale invariant curvature quantity for large $p$ and consequently obtain a global $L^{\infty}$ estimate using the Moser iteration.  This also can be achieved from the De Giorgi method.   For the rescaled flow, we also acquire a uniform curvature bound. Finally, in Section \ref{sec-asymptotic}, the proofs of Theorem \ref{thm-main-1} and Theorem \ref{thm-main-2} are given: the existence of a smooth limit manifold $M^*$ and the convergence of the rescaled flow to a round sphere under the conditions in Theorem \ref{thm-main-2}.  Moreover, in dimension two, by rescaling the hypersurfaces homothetically so that the volume is preserved, it is shown that a soliton equation is realized by the limit profile.

\section{Evolution Equations}\label{sec-evo}

In this section, we obtain the evolution of the quantities related to the curvature of the hypersurface $M_t$.
Prior to the computation, we introduce the first and the second derivatives of the curvature $F$ with respect to $h_{ij}$ :
 \begin{align}\label{F_dots}
  \dot{F}^{ij} &:= \psi R^{-1/2}(H g^{ij} - h^{ij}),
  \\
  \ddot{F}^{ij,kl} &:= -\psi R^{-3/2}(H g^{ij} - h^{ij})(H g^{kl} - h^{kl})+\psi R^{-1/2}(g^{ij} g^{kl} -\delta^{ik} \delta^{jl}) . \nonumber
 \end{align}
\begin{lemma}\label{ev-h}
Under the parabolic flow (\ref{eq-main}), we have
 \begin{enumerate}
\item $\displaystyle \frac{\partial}{\partial t}g_{ij}=-2F h_{ij} $
\item $\displaystyle \frac{\partial}{\partial t}\nu =\D ^{j} F \frac{\partial X}{\partial x^j}$
\item $\displaystyle \frac{\partial}{\partial t} h_{ij}=\D_{i}\D_{j}F -F  h_{jl}h^{l}_{i} $
\item $\displaystyle \frac{\partial}{\partial t}H = 2F h^{ij}h_{ij} + g^{ij} \frac{\partial}{\partial t}h_{ij} $
\item $\displaystyle \frac{\partial}{\partial t}|A|^2 = 2(\nabla_i \nabla_j F)h^{ij}+2F (tr A^3) $
\item $\displaystyle \frac{\partial}{\partial t}F = \dot{F}^{ij} \nabla_i \nabla_j F + \psi^2 H|A|^2- \psi^2 (tr A^3) -\psi R ( D\psi \cdot \bnu ) $,
 \end{enumerate}
where $tr A^3 = h_{ij} h^{jk} h_k^i$, and $D$ is the gradient in $\R^{n+1}$.
\end{lemma}
\begin{remark}
Due to the term $-\psi R (D\psi \cdot \bnu)$, a lower bound for $F$ may not follow directly from Lemma \ref{ev-h} (6) in general.
\end{remark}
From Lemma \ref{ev-h}, the detailed evolutions of the second fundamental form and hence the mean curvature can be derived.
\begin{lemma}\label{ev-h1}
Under the anisotropic scalar curvature flow (\ref{eq-main}), we have
 \begin{align*}
  \frac{\partial}{\partial t}h_{ij}
  &=\dot{F}^{kl} \D_{k}\D_{l} h_{ij}
  -\frac{\psi R^{-\frac{1}{2}}}{H^2}(H\nabla_i h_{kl}-h_{kl}\nabla_i H)(H\nabla_j h_{kl}-h_{kl}\nabla_j H)
 \\
  &-\frac{1}{4} \psi H^4 R^{-\frac{3}{2}} \nabla_i \left(\frac{|A|^2}{H^2} \right) \nabla_j \left(\frac{|A|^2}{H^2} \right)
 \\
  &+R^{-\frac{1}{2}}\left[(H\nabla_j H - A \cdot \nabla_j A) \cdot \nabla_i \psi +(H\nabla_i H-A\cdot \nabla_i A)\cdot \nabla_j \psi \right]
 \\
  &+ R^{\frac12}\nabla_i \nabla_j \psi
   + \psi R^{-\frac{1}{2}}[(H|A|^2-tr A^3)h_{ij}-2R(A^2)_{ij}]
  \end{align*}
where $tr A^3 = h^{ij}h_{jk} h_i^k$.  Let $\Box:=(Hg^{ij}-h^{ij}) \nabla_i \nabla_j$.
  \begin{align*}
    \displaystyle \frac{\partial}{\partial t}H
    &= R^{-1/2}[ \psi \Box H - \frac{\psi}{H^2}|H\nabla h_{kl}-h_{kl}\nabla H|^2 -\frac{\psi}{4R}|H^2\nabla H + \frac{2R}{\psi}\nabla \psi|^2
    \\
    &+ \frac{2R}{H}\langle \nabla H, \nabla \psi \rangle ]+ \frac{R^{1/2}}{\psi}|\nabla \psi|^2+R^{1/2}\Delta \psi + \psi R^{-1/2}H[H|A|^2-tr A^3] ,
  \end{align*}
 \begin{align*}
  \displaystyle \frac{\partial}{\partial t}H^2
&=\dot{F}^{kl} \D_{k}\D_{l} H^2 - \frac{2\psi}{R^{1/2}}|\nabla H|^2_{Hg-h}
-\frac{2\psi}{HR^{1/2}}|H\nabla h_{kl}-h_{kl} \nabla H|^2 + \frac{2H^2\psi}{R^{1/2}}[H|A|^2-tr A^3]
\\
&-\frac{H\psi}{2R^{3/2}}|H^2\nabla \Big(\frac{|A|^2}{H^2}\Big)+\frac{2R}{\psi}\nabla \psi|^2
+4R^{1/2}\nabla_i H \cdot \nabla_i \psi +2HR^{1/2}\Delta \psi +\frac{2HR^{1/2}}{\psi}|\nabla \psi|^2 ,
 \end{align*}
 \begin{align*}
\displaystyle \frac{\partial}{\partial t}|A|^2
&= \dot{F}^{kl}\nabla_k \nabla_l |A|^2 -\frac{2\psi}{R^{1/2}}|\nabla h_{ij}|^2_{Hg-h} -\frac{2\psi}{H^2 R^{1/2}}|H\nabla h_{kl}-h_{kl}\nabla H|^2_h
-\frac{H^4 \psi}{2 R^{3/2}} \Big|\nabla \Big(\frac{|A|^2}{H^2} \Big) \Big|^2_h
\\
&+\frac{2\psi}{R^{1/2}}(H|A|^2-tr A^3)|A|^2 +\frac{4h^{ij}}{R^{1/2}} (H\nabla_j H - A \cdot \nabla_j A) \cdot \nabla_i \psi
+2R^{1/2} h^{ij} \nabla_i \nabla_j \psi ,
 \end{align*}
 \begin{align*}
\frac{\partial}{\partial t} \Big( \frac{|A|^2}{H^2} \Big) &= \psi R^{-1/2} \Box \Big( \frac{|A|^2}{H^2} \Big)
-\frac{2\psi R^{1/2}}{H^5}|H\nabla h_{kl} -h_{kl}\nabla H |^2
+ \frac{2\psi}{R^{1/2}H} \langle \nabla H, \nabla \Big(\frac{|A|^2}{H^2} \Big) \rangle_{Hg-h}
\\
&+ \frac{\psi H}{2R^{3/2}}|\nabla \Big(\frac{|A|^2}{H^2}\Big)+\frac{2R}{\psi H^2}\nabla \psi|^2_{|A|^2g-Hh} -\frac{2R^{1/2}}{\psi H^3}|\nabla \psi|^2_{|A|^2g-Hh}
\\
&-\frac{4R^{1/2}}{H^4} \langle \nabla H, \nabla \psi \rangle_{|A|^2g-Hh} - \frac{2R^{1/2}}{H^3}(|A|^2g^{ij}-H h^{ij})\nabla_i \nabla_j \psi.
 \end{align*}
\end{lemma}

\section{Pinching Estimate}\label{sec-pinching}

In this section, it shall be shown that if all the principal curvatures are of the same order at the initial time, it remains so until the maximal time.  This will be used in later sections to deduce the convergence to a point and find the limit of the rescaled solution.  
The pinching estimate in more general setting was obtained in the earlier work \cite{KK}, and for the convenience of the reader, we give the detail of the computation for our case.
Denote $E=R^{1/2}$ so that $F=\psi E$.  With the notation \eqref{F_dots}, we have
 \begin{align}\label{E_dots}
   \dot{F}^{ij} &= \psi \dot{E}^{ij}  ,
      \qquad
    \ddot{F}^{\,ij,\,kl} = \psi \ddot{E}^{\,ij,\,kl}.
 \end{align}
Define a $(2,0)$-tensor $W$ by
 \begin{eqnarray*}
 W_{ij} &=& h_{ij} - \epsilon ( H + c ) g_{ij},
 \end{eqnarray*}
 for some constant $c>0$ depending only on $n, \ep$, derivatives of $\psi$ and initial data to be specified later in this section.
A simple computation yields
 \begin{align}\label{W_ij_time_deriv}
  \begin{split}
  \frac{\partial g_{ij}}{ \partial t} &= -2\psi E h_{ij} 
  \\
  \frac{\partial h_{ij}}{ \partial t} 
&= \psi[\dot{E}^{\, kl} \nabla_k \nabla_l h_{ij} + \ddot{E}^{\, kl, \, pq} \nabla_i h_{kl} \nabla_j h_{pq} + h_{ij} \dot{E}^{\, kl} g^{pq} h_{kp} h_{lq} - 2E g^{kl} h_{ik} h_{jl}  ]
\\
&+ E \nabla_i \nabla_j \psi + \nabla_i E \nabla_j \psi + \nabla_i \psi \nabla_j E
\\
\frac{\partial W_{ij}}{ \partial t} &= \psi \dot{E}^{\, kl} \nabla_k \nabla_l W_{ij}
+ W_{ij} \psi \dot{E}^{kl} g^{pq} h_{kp} h_{lq} -2W_{jl} \psi E g^{kl} h_{ik}
+ N_{ij} +  M_{ij}
\end{split}
\end{align}
where
\begin{align*}
h_{ijk} &=\nabla_k h_{ij} \, ,
\\
N_{ij} &= \psi \ddot{E}^{\, kl, \, pq}[h_{ikl} h_{jpq}-\epsilon g_{ij} g^{rs}h_{rkl} h_{spq}]
        +c\epsilon \psi g_{ij}\dot{E}^{kl} g^{pq} h_{kp} h_{lq} \, ,
\\
M_{ij} &= E[\nabla_i \nabla_j \psi- \epsilon g_{ij} \Delta \psi]+\dot{E}^{\, kl}[ h_{ikl} \nabla_j \psi + h_{jkl} \nabla_i \psi  -2\epsilon g_{ij}g^{pq} h_{pkl} \nabla_q \psi].
\end{align*}
For a symmetric matrix $A$, we may write $E(A)=f(\lambda(A))$, where $\lambda(A)=(\lambda_1,\cdots,\lambda_n)$ is the map which takes $A$ to its eigenvalues $\lambda_j$.  Thus for $A=(h_{ij})$, one has
 \begin{align*}
   E(h_{ij}) &=f(\lambda_1,\cdots,\lambda_n)= R^{1/2}= \Big(\sum_{i\neq j}\lambda_i\lambda_j \Big)^{1/2}.
 \end{align*}
 \begin{definition}
 For a $C^2$ function $F$ defined on the cone $S_+$ of positive definite symmetric matrices, we say $F$ is inverse-concave if
 \begin{align*}
   F^*(A) &:= -f(\lambda_1^{-1},\cdots,\lambda_n^{-1}),
  \end{align*}
  is concave for any $A \in S_+$.
 \end{definition}
For a symmetric matrix $A$ with eigenvalues $\lambda_i$'s, $R^{1/2}(A)$, a symmetric homogeneous function of degree one, is concave and inverse concave since the ratio of symmetric functions $\sigma_{k+1} / \sigma_k$, $k=0,\cdots,n-1$, and their geometric means are concave and inverse-concave as shown in \cite{A6,Lib}.
Let $\dot{f}^k$ denote the derivative of $f(\lambda(A))$ with respect to $\lambda_k$, and $(\delta_{ij})$ be the diagonal matrix with $1$ in the entries.  From the definition of inverse concavity,
one can obtain the following lemma:
 \begin{lemma}[Corollary 5.4 \cite{A6}]\label{lem_fconcave}
Let $A$ be a symmetric $n\times n$ matrix and let $F=F(A)$ be a smooth function of $A$. Then
$F^*$ is concave at $A$ if and only if $\left( \frac{\dot{f}^k-\dot{f}^l}{\lambda_k-\lambda_l}+\frac{\dot{f}^k}{\lambda_l}+\frac{\dot{f}^l}{\lambda_k}\right) \ge 0$ for all $k\neq l$, and $ \left(\ddot{f}^{\,kl}+2\frac{\dot{f}^k}{\lambda_k}\delta_{kl}\right)\ge 0$.
\end{lemma}
Indeed this holds for $A=(h_{ij})$ and $f=R^{1/2}$ whose derivatives are
 \begin{align}
 \fdi &= R^{-1/2} (H-\hii) \, ,  \label{f_1st_deriv}
 \\
 \fij &= -R^{-3/2}(H-\hii)(H-\hjj) + \psi R^{-1/2}(1-\delta^{ij}) . \nonumber
 \end{align}
In order to apply the maximum principle for $W_{ij}$, one shall see in Lemma \ref{null_lemma} that a perturbation term $c\ep g_{ij}$ in $W_{ij}$ nullifies the effect of $\psi$.  Assuming this, one obtains the following pinching estimate adapting Theorem 3.2 in \cite{A6}.
  \begin{theorem}\label{thm-pinching}
  Let
 \begin{align}\label{curvature_pinching}
  W_{ij}:= h_{ij}-\epsilon (H+c)g_{ij},
 \end{align}
 for some constants $\epsilon>0$ and $c>0$ satisfying
 \begin{align}\label{c_condition}
   c &\ge \max \Big\{ \fr{5 n^2 ||D \psi ||}{\ep^2 \psi} \, , \, \fr{3n ||D^2 \psi||^{1/2}}{\ep  \psi^{1/2}} + \fr{2n^3 ||D\psi ||}{\psi} \Big \} ,
 \end{align}
where $D$ is the gradient in $\R^{n+1}$.
Then if $W_{ij}$ is non-negative everywhere in $M$ at time $t=0$, then it remains so on $M \times [0,T]$.
\end{theorem}

\begin{proof}
Suppose that $W_{ij}$ takes its minimum at $(p, t_0) \in M \times [0, T]$ in the direction say $v \in T_p M$ where local coordinates $\{x_1, \cdots, x_n \}$ are chosen to have $v=\fr{\partial}{\partial x^1}$ and the connection coefficients vanish at $(p, t_0)$.  Taking $p=0$ for convenience, one has
 \begin{align*}
  \min_{(x,t)} \, \min_{\bxi} W_{ij}\xi^i \xi^j &= W_{11}(0 , t_0) =0.
 \end{align*}
Then for any $n\times n$ matrix $B^{ij}=B^{ij}(x_1, \cdots, x_n)$ and $\bxi= \xi^i \fr{\partial}{\partial x^i}$ where $\xi^i=\delta^i_1 + B^{ij}x_j$,
 \begin{align*}
  W(x,t):= W_{ij} \xi^i \xi^j (x,t) &\ge 0 \quad \textrm{for $t\in [0,t_0]$},
 \end{align*}
satisfies
 $
 W(0, t_0) = W_{11}(0,t_0)=0.
 $
At $(0, t_0)$, since $\xi^i(0)=\delta^i_1$ and $\fr{\pa \xi^i}{\pa x^k}=B^{ik}$, one has
 \begin{align*}
  \fr{\pa W}{\pa x^k}  
  &= \fr{\pa W_{11}}{\pa x^k} + 2W_{i1} B^{ik} =0
  \\
  \fr{\pa^2 W}{\pa x^k \pa x^l} 
  &= \fr{\pa^2 W_{11}}{\pa x^k \pa x^l}+ 2\fr{\pa W_{j1}}{\pa x^k} B^{jl} + 2\fr{\pa W_{i1}}{\pa x^l}B^{ik} + 2W_{ij}B^{ik}B^{jl}
 \end{align*}
so that for $i>1$,
\begin{align*}
  \dot{E}^{kl} \nabla_k \nabla_l W 
 &= \dot{E}^{\,kl} \fr{\pa^2 W_{11}}{\pa x^k \pa x^l} -\fr{2\dot{f}^{\,k}}{h_{ii}-h_{11}} \delta_{kl} \delta_{ij} \fr{\pa W_{j1}}{\pa x^k} \fr{\pa W_{i1}}{\pa x^l}
 \\
 &+ 2(h_{ii}-h_{11}) \dot{f}^{\,k} \delta_{kl} \delta_{ij} \Big(B^{ik}+\fr{1}{h_{ii}-h_{11}} \fr{\pa W_{i1}}{\pa x^k} \Big) \Big( B^{jl}+\fr{1}{h_{jj}-h_{11}}  \fr{\pa W_{j1}}{\pa x^l} \Big) . 
  \end{align*}
By taking $B^{ik}=-\fr{1}{h_{ii}-h_{11}} \fr{\pa W_{i1}}{\pa x^k}$, one obtains
\begin{align*}
\dot{E}^{\,kl} \fr{\pa^2 W_{11}}{\pa x^k \pa x^l} (0, t_0) &\ge \fr{2\dot{f}^{\,k}}{h_{ii}-h_{11}} h_{1ik}^2 \, ,
\end{align*}
where $i>1$ since $\fr{\pa W_{i1}}{\pa x^k} =h_{1ik}-\ep h_{ppk} \delta_{1i}=h_{1ik}(1- \delta_{i1})$ at $(0,t_0)$.
From \eqref{W_ij_time_deriv}, at $(0,t_0)$,
 \begin{align}\label{W_time_deriv_RHS}
  \fr{\pa W}{\pa t} 
  &\ge \fr{2\psi \dot{f}^{\,k}}{h_{ii}-h_{11}} h_{1ik}^2 + N_{11} + M_{11} \, ,
 \end{align}
where
 \begin{align*}
N_{11} &= \psi \ddot{E}^{\, kl, \, pq}[h_{1kl} h_{1pq}-\epsilon h_{rkl} h_{rpq}]
        +c\epsilon \psi \dot{E}^{kl} g^{pq} h_{kp} h_{lq} \, ,
\\
M_{11} &= E[\nabla_1 \nabla_1 \psi- \epsilon \Delta \psi] + \dot{E}^{\, kl}[ h_{1kl} \nabla_1 \psi + h_{1kl} \nabla_1 \psi  -2\epsilon h_{pkl} \nabla_p \psi]\, .
\end{align*}
By applying a maximum principle, it suffices to show that the right hand side of \eqref{W_time_deriv_RHS} is non-negative.  The conclusion follows from Lemma \ref{null_lemma} below which provides the required null eigenvector condition.
 \end{proof}
 \begin{lemma}\label{null_lemma}
Suppose that $c$ is chosen to satisfy
 \begin{align}\label{c_condition2}
   c &\ge \sup \Big\{ \fr{5 n^2 ||D \psi ||}{\ep^2 \psi} \, , \, \fr{3n ||D^2 \psi||^{1/2}}{\ep  \psi^{1/2}} + \fr{2n^3 ||D\psi ||}{\psi} \Big \} ,
 \end{align}
and that local coordinates are taken as in Theorem \ref{thm-pinching}.  Then one has
 \begin{align}\label{null_lem}
 \fr{2\psi \dot{f}^{\,k}}{h_{ii}-h_{11}} h_{1ik}^2 + N_{11} + M_{11} &\ge 0 \quad \textrm{at $(0, t_0)$}.
\end{align}
 \end{lemma}
 \begin{proof}
It suffices to consider the case in which $A=(h_{ij})$ has distinct principal curvatures $h_{ii}$, $i=1,\cdots,n$, as one can take a sequence of perturbed $A$ with distinct eigenvalues converging to that with repeated eigenvalues.  The computation in this lemma is carried at the point $(0,t_0)$ of minimum of $W$ defined in Theorem \ref{thm-pinching}.
With respect to an orthonormal frame $\{e_1, \cdots, e_n\}$ of eigenvectors for $A$, choosing $v=e_1$, with the corresponding eigenvalues $\{ h_{11}, \cdots, h_{nn}  \}$ arranged in ascending order, one has $A=diag(h_{11}, \cdots, h_{nn} )$ and $h_{11}=\epsilon (H+c)$.  We denote $\psi_k$ for $\nabla_k \psi$ and $\psi_{ij}$ for $\nabla_i \nabla_j \psi$ for short.

Considering the first and the second derivatives of the smooth map $Z:Sym(n) \times \mathbb{R}^n \times O(n) \longrightarrow Sym(n)$ given by $Z(A, \lambda, M)=M^t A M - diag(\lambda)$, one can show that for a symmetric $n\times n$ matrix $B$,
 \begin{align*}
  \ddot{F}(B,B) &= \psi \sum_{k,l}  \ddot{f}^{\,kl}B_{kk} B_{ll} + 2 \psi \sum_{k<l} \frac{\dot{f}^k-\dot{f}^l}{\hk -\hl}B_{kl}^2,
 \end{align*}
where its proof can be found in Theorem 5.1 in \cite{A6}.  This implies that \eqref{null_lem} can be written as
 \begin{align*}
  Q &:= \sum_{k, l \ge 1}[ \ddot{F}(h_{1kl}, h_{1kl})-\epsilon \sum_{j \ge 1} \ddot{F}(h_{jkl}, h_{jkl}) ]
  + 2\psi \sum_{\substack{k \ge 1 \\ l>1}}  \frac{\dot{f}^k}{\hl - \hone} h_{1kl}^2
  + R^{1/2}[\psi_{11} -\ep \Delta \psi]
  \\
  &+ 2 \sum_{j \ge 1} \dot{f}^{\, j} [h_{1jj} \psi_1 - \epsilon \sum_{k \ge 1} h_{kjj} \psi_k] + c \epsilon \psi \sum_{k \ge 1} \dot{f}^{\, k} h_{kk}^2
  \\
  &= \psi \sum_{k,l\ge1} \ddot{f}^{\,kl}h_{1kk} h_{1ll}
  +2 \psi \sum_{k<l} \frac{\dot{f}^k-\dot{f}^l}{\hk -\hl} h_{1kl}^2
  -\epsilon \psi \sum_{j,k,l \ge 1} \ddot{f}^{\,kl}h_{jkk}h_{jll }
  \\
  &- 2 \ep \psi \sum_{\substack{j \ge 1 \\ 1 \le k<l}} \frac{\dot{f}^k-\dot{f}^l}{\hk-\hl} h_{jkl}^2
   + 2 \psi \sum_{\substack{k \ge 1 \\ l >1}} \frac{\dot{f}^k}{\hl- \hone}h_{1kl}^2
   + R^{1/2}[\psi_{11} -\ep \Delta \psi ]
   \\
   &+ 2 \sum_{j \ge 1}\dot{f}^{\, j} [h_{1jj} \psi_1  - \epsilon \sum_{i \ge 1} h_{ijj} \psi_i]
    +c \epsilon \psi \sum_{k \ge 1} \dot{f}^{\, k} h_{kk}^2 \, .
 \end{align*}
 Since $W_{11}(0, t_0)=0$, one has
 \begin{align}\label{linear_relation}
  h_{k11} &= \frac{\epsilon}{1-\epsilon}\sum_{j>1}h_{kjj}, \quad k \ge 1,
 \end{align}
 which shall be used frequently in the computation below.  The sum $Q$ of quadratic terms can be decomposed into $Q_k$, $k \ge 1$, ones with repeated indices, and $Q_{jkl}$, ones with distinct indices:
 \begin{align*}
  Q &= Q_0 + \sum_{1\le k \le n} Q_k + \sum_{1\le j<k<l\le n} Q_{jkl},
 \end{align*}
where
   \begin{align*}
   Q_0 &= R^{1/2}[\psi_{11}- \epsilon \Delta \psi] + c \epsilon \psi \sum_{1 \le k  < n} \dot{f}^{\, k} h_{kk}^2 + \fr{1}{2} c \epsilon \psi \dot{f}^{\, n} \hn^2 \, ,
   \\
   Q_1 &= (1-\ep)\psi \sum_{i,j \ge 1} \ddot{f}^{\,ij} h_{1ii} h_{1jj} + 2\psi \sum_{j>1}\frac{\dot{f}^j}{\hj-  \hone}h_{1jj}^2
   -2\ep \psi \sum_{j>1} \frac{\dot{f}^j-\dot{f}^1}{\hj -\hone}h_{1jj}^2
   \\
   &+ 2(1-\ep) \sum_{\substack{j\ge 1}} \psi_1 \dot{f}^{\, j} h_{1jj}
     + \fr{1}{2n} c \epsilon \psi \dot{f}^{\, n} \hn^2 \, ,
   \end{align*}
   \begin{align*}
    Q_k
    &=-\epsilon \psi \sum_{i,j \ge 1} \ddot{f}^{\,ij} h_{kii} h_{kjj} +  2 \psi \frac{\dot{f}^1}{\hk- \hone}h_{k11}^2 +2 \psi \frac{\dot{f}^k-\dot{f}^1}{\hk- \hone}h_{k11}^2
     \\
    &-2\epsilon \psi \sum_{\substack{j\ge 1 \\ j \neq k}} \frac{\dot{f}^j-\dot{f}^k}{\hj -\hk}h_{kjj}^2
     -2\ep \sum_{\substack{j\ge 1}} \psi_k \dot{f}^{\, j} h_{kjj} + \fr{1}{2n} c \epsilon \psi \dot{f}^{\, n} h_{nn}^2 \, , \quad \textrm{for $k >1$}
  \end{align*}
  \begin{align*}
  Q_{1kl}
    &=2 \Big[  (1-\epsilon) \frac{\dot{f}^k-\dot{f}^l}{\hk-\hl}+ \Big( \frac{\dot{f}^k}{\hl-  \hone}
   + \frac{\dot{f}^l}{\hk -  \hone} \Big)
    - \ep \Big( \frac{\dot{f}^k-\dot{f}^1}{\hk-\hone} + \frac{\dot{f}^l-\dot{f}^1}{\hl-\hone} \Big) \Big]  h_{1kl}^2 \,  ,
 \\
  Q_{jkl} &= -2\epsilon \Big[ \frac{\dot{f}^j-\dot{f}^k}{\hj-\hk} + \frac{\dot{f}^k-\dot{f}^l}{\hk-\hl} + \frac{\dot{f}^l-\dot{f}^j}{\hl-\hj}\Big] h_{jkl}^2 \, , \quad 1<j<k<l\le n.
 \end{align*}
We will show that each of these is non-negative with appropriate choice of $c$.
For the computation below, 
we use
the following estimates which can be easily obtained from the definition $\dot{f}^i$ and the pinching condition at $(0, t_0)$:
 \begin{align*}
     & \frac{\dot{f}^i-\dot{f}^k}{\hi -\hk} = -R^{-1/2} ,
     \\
     &
    h_{ii} \ge \ep(h_{jj} + c) \ge c \ep  \qquad \textrm{for $i, j=1,\cdots,n$} ,
      \\
     & \fr{\ep^{1/2}}{n} \le \fr{h_{n-1 \, n-1}}{n (\hn h_{n-1\, n-1})^{1/2}} \le \fdi \le \fr{ h_{nn}}{\hone} \le \ep^{-1} .
 \end{align*}
(i)  From Lemma \ref{lem_fconcave}, it follows that $Q_{1kl} \ge 0$.
 \\
(ii) From the concavity of $f$ and Lemma \ref{lem_fconcave}, one has $Q_{jkl} \ge 0$ for $1<j<k<l$.
\\
(iii) For a fixed $k>1$,
 \begin{align*}
    Q_k
    &\ge 2 \psi \frac{\dot{f}^k}{\hk}h_{k11}^2
        +2\epsilon \psi \sum_{\substack{j \ge 1 \\ j \neq k}} R^{-1/2}h_{kjj}^2
        -2\ep \sum_{\substack{j\ge 1}} \psi_k \dot{f}^{\, j} h_{kjj} + \fr{1}{2n} c \epsilon \psi \dot{f}^{\, n} h_{nn}^2 .
 \end{align*}
By Young's inequality and \eqref{linear_relation},
\begin{align*}
  \sum_{j \ge 1} \psi_k \dot{f}^{\, j} h_{kjj} 
  &\le \sum_{\substack{j\ge 1 \\ j \neq k}} \fr{\psi}{2R^{1/2}} h_{kjj}^2
  + \fr{2n |\psi_k|^2}{\psi} |\dot{f}^{\, j}|^2 R^{1/2}
  + \psi_k \dot{f}^{\, k} \big( \fr{1-\ep}{\ep}h_{k11} -\sum_{\substack{j>1 \\ j \neq k}} h_{kjj} \big)
  \\
  &\le  \sum_{\substack{j\ge 1 \\ j \neq k}} \fr{\psi}{2R^{1/2}} h_{kjj}^2
  + \fr{2n |\psi_k|^2}{\ep^2 \psi} R^{1/2}
  + \psi \fr{\dot{f}^k}{\ep \hk} h_{k11}^2 + \fr{ \hk |\psi_k|^2}{\ep \psi} \dot{f}^{\, k}
  \\
  &+ \sum_{\substack{j>1 \\ j \neq k}} \fr{\psi}{2R^{1/2}} h_{kjj}^2
   + \fr{2n |\psi_k|^2 }{\psi }|\dot{f}^{\, k}|^2  R^{1/2}
  \\
  &\le  \sum_{\substack{j \ge 1 \\ j \neq k}} \fr{\psi}{R^{1/2}} h_{kjj}^2
  + \psi \fr{\dot{f}^k}{\ep \hk} h_{k11}^2
  + \fr{|\psi_k|^2}{\ep^2 \psi} H
  + \fr{4n |\psi_k|^2}{\ep^2 \psi} R^{1/2} \, ,
  \end{align*}
which yields
 \begin{align*}
   Q_k &\ge -2 \fr{|\psi_k|^2}{\ep \psi} H
            - \fr{8n |\psi_k|^2}{\ep \psi} R^{1/2}
            + \fr{1}{2n} c \epsilon \psi \dot{f}^{\, n} h_{nn}^2
        \ge - \fr{10n |\psi_k|^2}{\ep \psi} H
            + \fr{1}{2n^4} c \epsilon^{3/2} \psi H^2 .
 \end{align*}
Therefore since $H \ge c n\ep$, $Q_k \ge 0$ if
 \begin{align}\label{c_condition_Q_k}
c &\ge \fr{5n^2 || D \psi ||}{\ep^{7/4} \psi} \, .
 \end{align}
\newline
(iv) Consider $Q_1$:
\begin{align*}
   Q_1 &= (1-\ep)\psi \sum_{i,j \ge 1} \ddot{f}^{\,ij} h_{1ii} h_{1jj}
        + 2\psi \sum_{j>1}\frac{\dot{f}^j}{\hj-  \hone}h_{1jj}^2
        - 2\ep \psi \sum_{j>1} \frac{\dot{f}^j-\dot{f}^1}{\hj -\hone}h_{1jj}^2
   \\
    &+ 2(1-\ep) \sum_{\substack{j\ge 1}} \psi_1 \dot{f}^{\, j} h_{1jj}
     + \fr{1}{2n} c \epsilon \psi \dot{f}^{\, n} h_{nn}^2 \, ,
   \end{align*}
Using \eqref{linear_relation}, the first sum above can be written as
\begin{align*}
(1-\epsilon)\sum_{k,l> 1} \Big(\ddot{f}^{\,kl}+\frac{\epsilon}{1-\epsilon}(\ddot{f}^{\,k1} + \ddot{f}^{\,l1})+\Big(\frac{\epsilon}{1-\epsilon}   \Big)^2 \ddot{f}^{\,11} \Big)h_{1kk} h_{1ll}.
\end{align*}
Let
 $   \phi(x_2,\cdots, x_n) := f(\frac{\epsilon}{1-\epsilon}(x_2+\cdots+x_n), x_2,\cdots, x_n) $ ,
and then its derivatives are
  \begin{align*}
   \dot{\phi}^k &= \dot{f}^k + \frac{\epsilon}{1-\epsilon}\dot{f}^1 ,
  \\
   \ddot{\phi}^{\,kl} &= \ddot{f}^{\,kl} + \frac{\epsilon}{1-\epsilon}(\dot{f}^{\,k1}+\dot{f}^{\,l1}) + \Big(\frac{\epsilon}{1-\epsilon}\Big)^2 \ddot{f}^{\,11}
  \end{align*}
for $k,l>1$.  The coefficients of the second and the third sum in $Q_1$ above are
   \begin{align*}
    2\psi \Big( \frac{\dot{f}^k}{\hk- \hone} - \epsilon  \frac{\dot{f}^k-\dot{f}^1}{\hk -\hone} \Big)
     &\ge \frac{2\psi}{\hk- \hone} [(1-\ep) \dot{f}^k + \ep \dot{f}^1]
      = 2(1-\ep) \psi \frac{\dot{\phi}^k}{\hk - \hone},
   \end{align*}
   so that
   \begin{align*}
    Q_1 &\ge (1-\epsilon) \psi \sum_{k,l>1} (\ddot{\phi}^{\,kl}+\frac{2\dot{\phi}^k}{\hk-\hone}\delta_{kl}) h_{1kk} h_{1ll}
    + 2(1-\ep) \sum_{\substack{j\ge 1}} \psi_1 \dot{f}^{\, j} h_{1jj}
     + \fr{c \epsilon }{2n} \psi \dot{f}^{\, n} h_{nn}^2
    \\
    &\ge 2 \psi \sum_{k>1} \big( \fr{1}{\hk-\hone} -\fr{1}{\hk} \big) \big[(1-\ep) \fk + \ep \fone \big] h_{1kk}^2
    + 2(1-\ep) \sum_{\substack{j\ge 1}} \psi_1 \dot{f}^{\, j} h_{1jj}
    \\
    &+ \fr{c \epsilon }{2n} \psi \dot{f}^{\, n} h_{nn}^2 \, ,
          \end{align*}
where the last inequality can be shown from the fact that the inverse-concavity of $\phi$ follows from that of $f$ and Lemma \ref{lem_fconcave} (ii).
By Young's inequality and \eqref{linear_relation} for the terms involving $\psi_1$,
 \begin{align*}
 2(1-\ep) \sum_{\substack{j\ge 1}} \psi_1 \dot{f}^{\, j} h_{1jj}
 &\ge - \fr{2\psi}{R^{1/2}} \sum_{k>1} \fr{\hone h_{1kk}^2}{\hk(\hk-\hone)} H_k 
 -\fr{2|\psi_1|^2}{R^{1/2} \psi} \sum_{k>1} \fr{\hk (\hk-\hone)}{\hone} H_k ,
\end{align*}
where $H_k = H-\hk + \ep(\hk-\hone)$, one has
 \begin{align*}
   Q_1 &\ge -\fr{2|\psi_1|^2}{R^{1/2} \psi} \sum_{k>1} \fr{\hk (\hk-\hone)}{\hone} H_k
           + \fr{ c \epsilon \psi }{2n R^{1/2}}(H-\hn)\hn^2
       \\
       &\ge -\fr{2n|\psi_1|^2}{R^{1/2} \psi} \fr{\hn^2}{\hone} H + \fr{ c \epsilon \psi }{2n R^{1/2}} \hone \hn^2  ,
 \end{align*}
which implies $Q_1 \ge 0$ if
 \begin{align}\label{c_condition_Q_1}
   c \ge  \fr{2n^{3/2} ||D \psi ||}{\ep^2 \psi} .
 \end{align}
 \newline
(v) The terms involving $D^2\psi$ in $Q_0$ can be bounded below as
 \begin{align*}
   R^{1/2}(\nabla_1 \nabla_1 \psi -\ep \Delta \psi) &\ge -R^{1/2} (|\nabla_1 \nabla_1 \psi| + \ep |\Delta \psi|)
   \ge -R^{1/2} \big[ 3||D^2 \psi || +  \ep H |D\psi| \big]
 \end{align*}
since
$
|\Delta_M \psi| \le (n+1)|| D^2 \psi || + H|D\psi| .
$
Thus one has
 \begin{align*}
   Q_0 \ge  -R^{1/2} \big[ 3||D^2 \psi || +  \ep H |D\psi| \big] + \fr{c\ep \psi}{2R^{1/2}} (H-\hn)\hn^2 \, ,
 \end{align*}
and hence $Q_0 \ge 0$ if
 \begin{align*}
   c \ge \fr{2R}{\ep \psi (H-\hn)\hn^2} \big[ 3||D^2 \psi|| + \ep H |D\psi| \big].
 \end{align*}
 The estimate $  \fr{R}{(H-\hn)\hn^2} \le  \fr{n^3}{H}$
 implies that $Q_0 \ge 0$ if
 \begin{align}\label{c_condition_2nd_deriv}
  c \ge \fr{3n ||D^2 \psi||^{1/2}}{\ep \psi^{1/2}} + \fr{2n^3 || D\psi ||}{\psi}.
 \end{align}
Finally, from the conditions \eqref{c_condition_Q_k}, \eqref{c_condition_Q_1} and \eqref{c_condition_2nd_deriv}, one can conclude that $Q$ is non-negative if $c$ satisfies the condition \eqref{c_condition2}.
\end{proof}
\begin{corollary}\label{eigenvalue_pinching}
Let $\hn$ and $\hone$ be the largest and the smallest nonzero eigenvalues of the Weingarten map respectively.  With the conditions in Theorem \ref{thm-pinching}, we have the following curvature pinching:
  \begin{align}\label{principal_curv_pinching}
  c\epsilon^2 &\le \epsilon \hone \le \hn \le \frac{1}{\epsilon} \hone .
  \end{align}
In particular, $ F \ge c\epsilon \inf_M \psi$.
\end{corollary}

\section{Curvature Estimates}\label{sec-curvature}

In this section, we obtain the uniform upper bounds for curvatures of the evolving manifolds before the maximal time $T$. 
With the Gauss map $\boldsymbol{\nu}$, one can parametrise the hypersurface so that the support function $S$ is written as
$$
S(z,t)= \langle z, X(\boldsymbol{\nu}^{-1}(z),t) \rangle, \qquad \text{for $z \in \mS^n$ and $ t \in [0,T-\delta_0]$} ,
$$
 for some fixed $\delta_0$.
Along the flow (\ref{eq-main}),
$$
\frac{\partial S}{\partial t}=\left(z,  \frac{\partial \boldsymbol{\nu}^{-1}}{\partial t}\cdot \nabla X+ \frac{\partial X}{\partial t}\right)=-F.
$$
Let $\overline{g}$ be the standard metric on $S^n$ and $\overline{\nabla}$ be its connection.  Note that $\overline{g}$ is independent of $t$.
From the definition, one has
\begin{align}\label{overline_g}
\begin{split}
\overline{g}_{ij} &= h_{ik}h_{jl}g^{kl},  \qquad |A|^2 = g^{ij} \overline{g}_{ij} , 
\\
h_{ij} &= \overline{\nabla}_i \overline{\nabla}_j S + S\overline{g}_{ij} ,
\end{split}
\end{align}
and therefore,
\begin{align*}
\frac{\partial h_{ij}}{\partial t} &= \overline{\nabla}_i \overline{\nabla}_j \frac{\partial S}{\partial t} + \frac{\partial S}{\partial t} \overline{g}_{ij}.
\end{align*}
\begin{lemma}\label{lem_evol_F}
With the metric $\overline{g}_{ij}$, we have the following evolution equation for $F$:
\begin{align*}
\frac{\partial F}{\partial t} &=  \frac{\psi Hg^{ij}}{R^{\frac{1}{2}}} \overline{\nabla}_i \overline{\nabla}_j F + \psi^2 (H|A|^2-tr A^3) -\psi R (D\psi \cdot \bnu)  .
\end{align*}
\end{lemma}

\begin{proof}
Lemma \ref{ev-h} and \eqref{overline_g} give
\begin{align*}
H  
= \overline{g}_{kl} (h^{-1})^{kl}  ,
\qquad
\frac{\partial H}{\partial t}  
= g^{ij} (\overline{\nabla}_i \overline{\nabla}_j F + F \cdot \overline{g}_{ij}).
\end{align*}
From these and Lemma \ref{ev-h}, we can compute the evolution of $F=\psi R^{\frac12}$:
 \begin{align*}
  \frac{\partial F}{\partial t}
   &= \frac{\psi Hg^{ij}}{R^{\frac{1}{2}}} \overline{\nabla}_i \overline{\nabla}_j F + \psi^2 (H|A|^2-tr A^3) -\psi R (D\psi \cdot \bnu) .
\end{align*}
\end{proof}
Let $r_{in}$ and $r_{out}$ be the inner and the outer radii of $\Sigma_t$, respectively, and let $w(z):= S(z)+S(-z)$, $z \in S^n$, be the width on $z$-direction.  Then we have the maximum width $w_{\max}=w(z_+)$ and the minimum $w_{\min}=w(z_{-})$ for some $z_+$ and $z_-$ in $S^n$.
\begin{lemma}\label{lem-iso-w}
\item If the initial hypersurface is pinched as in Theorem \ref{thm-pinching}, we have
     \begin{enumerate}[(i)]
         \item
           $w_{\max} \le C w_{\min}$,
         \item
           $r_{out} \le \frac{w_{\max}}{\sqrt{2}}$ and $r_{in} \ge \frac{w_{\min}}{n+2}$,
         and in particular, $r_{out} \le C r_{in}$.
      \end{enumerate}
      for some positive constant  $C$.
\end{lemma}
\begin{proof}
  \begin{enumerate}[(i)]
    \item
  For the parametrization of a totally geodesic 2-sphere in $S^n$, choose a pair of angular variables $(\theta_+,\psi_+)$ mapped to
  $( \cos\theta_+\sin\psi_+ ,$ $\sin\theta_+\sin\psi_+ ,$ $\cos\psi_+ )$ such that $\psi_+=0$ is corresponding to $z_+$.  Expressing the second fundamental form $\Pi$ in terms of the support function $S$ in the $\psi_+$ direction, one has
    \begin{align*}
  w(z_+) &= S(z_+)+S(-z_+)=\int_{S^2} \Pi(\partial\psi_+,\partial\psi_+)d\mu_{S^2}.
    \end{align*}
See Theorem 5.1 in \cite{A4} for the detail.
Similarly, we also have $w(z_-)=S(z_-)+S(-z_-)=\int_{S^2} \Pi(\partial\psi_-,\partial\psi_-)d\mu_{S^2}$ for the corresponding parametrization $(\theta_-,\psi_-)$.  From Corollary \ref{eigenvalue_pinching}, we have $\Pi(\partial\psi_+,\partial\psi_+)\leq C \Pi (\partial\psi_-,\partial\psi_-)$ which implies $w_{\max}\leq Cw_{\min}$.
   \item
Consider the intersection of the hypersurface and its largest enclosed sphere.  Then the intersection has at most (n+2) elements.  For the detailed proof, see Lemma 5.4 in \cite{A4}.
   \end{enumerate}
   \end{proof}
    \begin{lemma}\label{H_upperbound}
Let $M_0$ be convex and let $\delta_0>0$ be a fixed constant.  Set $t_0=T-\delta_0$ and $\rho_0=\frac{1}{2}r_{in}(t_0)$.
Suppose that the pinching condition holds at $t=0$ with the condition \eqref{c_condition} on $c$ as in Theorem \ref{curvature_pinching}, and
$\ep \le \fr{5}{n}$.
Then there is a constant $C_F=C_F(t_0)>0$ such that
\begin{align}\label{F_upper}
   \sup_{x\in M ,\ 0 \le t \le t_0}F (x,t)
    \leq
    C_F = \max \left(\sup_{x\in M } F(x,0), \sup_{x\in M }\frac{C}{\rho_0}\right) \, ,
\end{align}
where $C$ is the constant given in Lemma \ref{lem-iso-w}.
\end{lemma}
\begin{proof}
Reparametrize the hypersurface to be defined on $S^n$.
Let $F_S=\frac{F}{S-\rho_0}$ where $\rho_0=\frac{1}{2}r_{in}(t_0)$ on $S^n \times [0,t_0]$, and suppose that $F_S$ takes its maximum at $(z_1,t_1)$ on $S^n \times [0,t_0]$.  Assume $t_1>0$ and unless stated otherwise, the computation below is carried out at $(z_1,t_1)$.
\begin{align}\label{1st_deriv_Fhat}
\begin{split}
\overline{\nabla}_i F_S &=  \frac{\overline{\nabla}_i F}{S-\rho_0} - \frac{ F\overline{\nabla}_i S}{(S-\rho_0)^2} = 0.
\\
\frac{\partial F_S}{\partial t} &= \frac{F_t}{S-\rho_0}-\frac{FS_t}{(S-\rho_0)^2} \ge 0 .
\end{split}
\end{align}
By Lemma \ref{lem_evol_F}, this implies that
\begin{align}\label{t-deriv_Fhat}
0 &\le (S-\rho_0)[Hg^{ij} \overline{\nabla}_i \overline{\nabla}_j F + F(H|A|^2-tr A^3)
- R^{3/2} (D\psi \cdot \bnu )]+ F R.
\end{align}
From (\ref{1st_deriv_Fhat}), one easily gets
  \begin{align*}%\label{2nd_deriv_Fhat}
0 \ge \overline{\nabla}_i \overline{\nabla}_j F_S
&= \frac{\overline{\nabla}_i \overline{\nabla}_j F}{S-\rho_0}-\frac{F\overline{\nabla}_i \overline{\nabla}_j S}{(S-\rho_0)^2}, \nonumber
  \end{align*}
and using \eqref{overline_g}, one then has
  \begin{align}\label{Fhat_inequality1}
0 &\ge (S-\rho_0)[Hg^{ij} \overline{\nabla}_i \overline{\nabla}_j F+ HF|A|^2]+ \rho_0 FH |A|^2 - H^2F.
  \end{align}
Combining \eqref{t-deriv_Fhat} and \eqref{Fhat_inequality1},
  \begin{align}\label{Fhat_inequality2}
H^2 F+ F R &\ge (S-\rho_0) [ F \cdot tr A^3
+ R^{3/2} (D\psi \cdot \bnu )] + \rho_0 FH |A|^2.
  \end{align}
On the other hand, one also has $H^2 F+ F R  \le 2H^2 F$, and
for $\ep \le \fr{5}{n}$,
 \begin{align*}
 & (S-\rho_0) [ F \cdot tr A^3 +  R^{3/2} (D\psi \cdot \bnu )] + \rho_0 FH |A|^2  \ge \frac{\rho_0}{n}FH^3 ,
 \end{align*}
since one has $c \ep^2 \psi \ge 5n^2 |D\psi|$ from \eqref{c_condition} and
 \begin{align*}
    F \cdot tr A^3 +  R^{3/2} (D\psi \cdot \bnu ) &\ge \psi c\ep \hn^3 -  n^3 |D\psi| \hn^3
 \end{align*}
 from the pinching condition.
Thus for $\ep \le \fr{5}{n}$, \eqref{Fhat_inequality2} yields
$
H(z_1, t_1) \le \frac{2n}{\rho_0}.
$
If $F$ takes its maximum at $(z_2, t_1)$, take $t_0=t_1$, and then one has
\begin{align*}
F(z_2, t_1) &\le \frac{S(z_2, t_1)-\rho_0}{S(z_1,t_1)-\rho_0} F(z_1, t_1) \le  \frac{2 r_{out}(t_1)}{r_{in}(t_1)}  F(z_1, t_1) \le C \, F(z_1, t_1) .
\end{align*}
from Lemma \ref{lem-iso-w}.
Since
$
F(z_1, t_1) \le \psi(z_1) H(z_1,t_1) \le \frac{2n}{\rho_0} \psi_1 ,
$
one has
\begin{align*}
F(z_2, t_1) &\le C r_{in}(t_1)^{-1} .
\end{align*}
If $F$ takes its maximum at $(z_2, t_2)$ where $t_2<t_1$, then reduce $t_0$ so that $F_S$ takes its maximum at $t_1=t_2$, and follow the argument above.
   \end{proof}
\begin{lemma}\label{H_upper}
Let $M_0$ be convex, and suppose that the initial hypersurface is pinched as in Theorem \ref{thm-pinching}.
Then there exists a positive constant $C=C(n,\epsilon, \psi, C_F)$ such that
\begin{align*}
H \le C r_{out}(t_0)^{-1}
\end{align*}
\end{lemma}
\begin{proof}
From \eqref{F_upper}, we have
\begin{align*}
\hn \hone \le R \le C_F^2 \psi_0^{-2},
\end{align*}
where $\hn$ and $\hone$ are the largest and the smallest principal curvatures, respectively,
and the pinching estimate \eqref{principal_curv_pinching} implies
  \begin{align*}
 H^2 &\le n^2 \hn^2 \le n^2 \epsilon^{-1} \hn \hone
 \le  n^2 C_F^2 \epsilon^{-1} \psi_0^{-2}.
  \end{align*}
Then the result follows from Lemma \ref{H_upperbound}.
\end{proof}
To obtain the short time existence of the flow \eqref{eq-main}, using a radial function $r(z,t)$, $z \in S^n$, we parametrize the hypersurface by
  \begin{align*}
     X(x,t) &= r(z,t) z,
  \end{align*}
  for $x=\pi^{-1}(z) \in M$, where $\pi: M^n \rightarrow S^n$ is the normalizing map.
 \begin{lemma}
  If the initial hypersurface is pinched as in Theorem \ref{thm-pinching}, then one has the short time existence for the flow \eqref{evolution_eq1}.
  \end{lemma}
  \begin{proof}
      One can compute that
        \begin{align*}
           g^{ij} &= r^{-2} \Big( \overline{g}^{\, ij} - \frac{\overline{\nabla}^i r \overline{\nabla}^j r}{r^2+|\overline{\nabla}r|^2} \Big) \, , 
         \qquad
           \nu = \frac{1}{\sqrt{r^2+|\overline{\nabla} r|^2}} (rz - \overline{g}^{\, ij} \overline{\nabla}^{\, i} z \cdot \overline{\nabla}^{\, j} r)
         \\
           h_{ij} &= \frac{1}{\sqrt{r^2+ |\overline{\nabla}r|^2}}(-r\overline{\nabla}_i \overline{\nabla}_j r + 2\overline{\nabla}_i r \overline{\nabla}_j r + r^2 \overline{g}_{ij}) ,
         \end{align*}
       which can be found in Chapter 3 in \cite{Z}.
Then
  $
   \frac{\pa r}{\pa t} = -\frac{F}{r}\sqrt{r^2+|\overline{\nabla} r|^2},
  $
and since $F= \psi [(g^{ij}h_{ij})^2- g^{ik} g^{jl} h_{ij} h_{kl}]^{1/2}$, we have
  \begin{align}\label{r_evolution}
  \begin{split}
     \frac{\pa r}{\pa t} &= -\frac{\psi}{r^3} \Big[\{(\overline{g}^{ij} - \frac{\on^i r \on^j r}{r^2+|\on r|^2})(-r\on^i \on^j r + 2\on^i r \on^j r +r^2 \overline{g}^{ij}) \}^2
     \\
   &- (\overline{g}^{ik}-\frac{\on^i r \on^k r}{r^2+|\on r|^2})
      (-r\on^i \on^j r + 2\on^i r \on^j r +r^2 \overline{g}^{ij})
      \\
   &\times \, (\overline{g}^{jl}-\frac{\on^j r \on^l r}{r^2+|\on r|^2})
      (-r\on^k \on^l r + 2\on^k r \on^l r +r^2 \overline{g}^{kl}) \Big]^{1/2} 
      \end{split}
  \end{align}
For $M^n=\mS^n$, we have the following evolution equation of the radial function:
 \begin{align*}
  -\frac{\sqrt{n^2-1}\psi_1}{r} &\le \frac{d r}{d t} =  -\frac{\sqrt{n^2-1}\psi}{r}
  \le -\frac{\sqrt{n^2-1} \psi_0}{r}
 \end{align*}
so that
 \begin{align}\label{r_bound}
  \Big[r(0)^2-C^-(T-t) \Big]^{1/2} &\le r \le \Big[r(0)^2-C^+(T-t) \Big]^{1/2},
 \end{align}
where $C^-:=2(n^2-1)^{1/2}\psi_1$ and $C^+:=2(n^2-1)^{1/2}\psi_0$.
Regarding the hypersurface $X(\cdot,t)$ as a graph in \eqref{r_evolution}, the short time existence follows from the standard parabolic theory.
  \end{proof}
  With the parabolic rescaling \eqref{X_parametrisation}, one also has the following typical regularity of the rescaled flow
as such is given for the mean curvature flow in Lemma 7.2 of  \cite{A4} and in Lemma 3.5 of \cite{Z}:
  \begin{lemma}\label{lem-iso-esti}
    \item
      If the initial hypersurface is pinched as in Theorem \ref{thm-pinching}, we have
       \begin{enumerate}[(i)]
        \item
          $  \tilde{C}^{-1} \leq \tilde{r}_{in}(\tau)\leq \tilde{r}_{out}(\tau)\leq\tilde{C}\quad\text{for all $\tau>0$},$ and
        \item
          $  \sup_{(x,\tau) \in\Sigma_0 \times [0,\infty)}\tilde{H}(x,\tau)\leq \tilde{C}$.
    \end{enumerate}
\end{lemma}
\begin{proof}
  \begin{enumerate}[(i)]
    \item
    For $t\rq{}\in(t,T)$, the hypersurface $X(\cdot,t\rq{})$ is enclosed by $\partial B_{r_+(t)}(p)$ by the maximum principle for some $p \in \R^{n+1}$ where $r_+(t\rq{}) \le \left(r_{out}^2(t)-C^+(t\rq{}-t)\right)^{1/2}$ from \eqref{r_bound}.
 Thus we have $r_{out}(t) \ge (C^+(t\rq{}-t))^{1/2}$ and, letting  $t\rq{}$ tend to $T$, we have
 $r_{out}(t) \ge \left(C^+(T-t)\right)^{1/2}$.  With the rescaling, from Lemma \ref{lem-iso-w},
 \begin{align*}
 \frac{C^+}{2} &\le \tilde{r}_{out}(\tau)^2 \le C \tilde{r}_{in}(\tau)^2
 \end{align*}
 On the other hand, suppose that $X(\cdot, t)$ encloses the ball $B_{r_-(t)}(p)$ for some $p \in \R^{n+1}$ where
  $$
  r_{in}(t\rq{}) \ge r_-(t\rq{}) \ge \left(r_{in}^2(t)-C^-(t\rq{}-t)\right)^{1/2}.
  $$
 Recall that $X(\cdot, t)$ shrinks to a point as $t \ra T$, one has
 $
 r_{in}^2(t) \le C^-(T-t) ,
 $
 and hence $\tilde{r}_{in}^2(\tau) \le \frac{1}{2}C^-$.
    \item
   From Lemma \ref{H_upperbound} and Lemma \ref{H_upper}, one has
   $H \le C r_{in}(t_0)^{-1} \le C r_{out}(t_0)^{-1}$. Also the parametrization in \eqref{X_parametrisation} gives
    \begin{align*}
         r_{out}(t_0)^{-1} &= \frac{\tilde{r}_{out}(\tau_0)^{-1}}{\sqrt{2(T-t_0)}} \le \frac{C}{\sqrt{2(T-t_0)}}.
    \end{align*}
   Thus one has
 $
     \tilde{H}(x,\tau) = \sqrt{2(T-t)} H(x,t)  \le C.
      $
 \end{enumerate}
\end{proof}

\section{Higher regularity of $|A|^2$}\label{higherA}

In this section, we show that the rescaled second fundamental form $\tilde{A}$ and its derivatives are bounded depending only on the initial data and the derivatives of $\psi$. That is,
 \begin{align*}
   |\tna^m \tilde{\bA}| \le C(n, m, \ep, \tilde{H}_0, \psi, |D\psi|, \cdots, |D^{m+2} \psi|) \quad \textrm{for all $m \in \N$}\, .
 \end{align*}
Recall that $\dot{F} = \psi R^{-1/2}(Hg-h)$ and $R \sim O(H)$ from the pinching estimate \eqref{principal_curv_pinching}.  
Let $A \ast B$ denote any tensor field which is a linear combination of tensor field formed by contracting tensors $A$ and $B$ with the metric.
Since
   \begin{align*}
     \dot{F} &= \psi R^{-1/2}(Hg-h) = \psi R^{-1/2} \bA   ,
     \qquad
     \na \dot{F} 
      = \fr{1}{H} \big( \psi \na \bA + \na \psi \bA \big) \, ,
     \end{align*}
    one can show that
   \begin{align*}
     \na^k \dot{F} 
     &= \sum_{i+i_1+\cdots+i_{2k+1}=k} R^{-k-\fr{1}{2}} \na^{i} \psi \ast \na^{i_1} \bA \ast \cdots\ast \na^{i_{2k+1}} \bA 
   \end{align*}
   for $m \in \Z_+$ , using an induction on $m$. 
 Since $  \na^2 \psi = D^2 \psi - (\bnu \cdot D\psi) \bA$ where $D$ is the gradient in $\R$, and
 \begin{align*}
 \na(\dot{F}^{k \, l} \na_k \na_l \na^{m-1}\bA) &= \dot{F}^{k \, l} \na_k \na_l \na^m\bA + \na \bA \ast \ast \bA \na^{m-1} \bA + \bA \ast \bA  \ast \na^m\bA
\  \\
  &
  + \fr{1}{H} \sum_{i+j=1} \na^{i} \psi \ast \na^{j} \bA \ast \na^{m+1} \bA \,  ,
 \end{align*}
 one can easily compute that
\begin{align}\label{higher}
\begin{split}
    \fr{\pa}{\pa t} \nabla^m \bA
    &=
    \dot{F}^{p \, q} \nabla_p \nabla_q \nabla^m \bA
    + \sum_{i+j+k+l=m} \na^i \psi\ast \na^j \bA \ast \na^k \bA \ast \na^l \bA + \sum_{i+j=m} \na^{i+2} \psi \ast \na^j \bA
    \\
    &+ \sum_{k=1}^{m+1} \fr{1}{H^k} \sum_{i + i_1+\cdots +i_k=k} \na^i \psi \ast \na^{m-k+2} \bA \ast \na^{i_1} \bA \ast \cdots \ast \na^{i_k} \bA
 \end{split}
 \end{align}
and therefore, one can obtain
\begin{lemma}\label{A_m_deriv}
Given the pinching estimate \eqref{principal_curv_pinching}, one has
\begin{align*}%\label{higher}
\begin{split}
   \fr{\pa}{\pa t} |\nabla^m \bA|^2
    &=
    \dot{F}^{k \, l} \nabla_k \nabla_l  |\nabla^m \bA|^2 -2\psi R^{-1/2} |\na^{m+1} \bA|^2_{Hg-h}
    \\
     &+ \sum_{k=1}^{m+1} \fr{1}{H^k} \sum_{i + i_1+\cdots +i_k=k} \na^i \psi \ast \na^{m-k+2} \bA \ast \na^{i_1} \bA \ast \cdots \ast \na^{i_k} \bA \ast \na^m \bA
    \\
    &+ \sum_{i+j+k+l=m} \na^i \psi\ast \na^j \bA \ast \na^k \bA \ast \na^l \bA  \ast \na^m \bA + \sum_{i+j=m} \na^{i+2} \psi \ast \na^j \bA \ast \na^m \bA \, .
 \end{split}
 \end{align*}
\end{lemma}
By rescaling the hypersurface with $\tilde{X}=\phi^{-1/2} X$ where $\phi(t)=2(T-t)$, as in \eqref{X_parametrisation},
one has
\begin{lemma}\label{highA}
Given the pinching estimate \eqref{principal_curv_pinching}, one has
 \begin{align*}
   |\tna^m \tilde{\bA}| \le C(n, m, \ep, \tilde{H}_0, \psi, |D\psi|, \cdots, |D^{m+2} \psi|) \quad \textrm{for all $m \in \N$}\, .
 \end{align*}
\end{lemma}
\begin{proof}
Since $|\tilde{\bA}|^2$ is bounded, we use an induction on $m$ so that we assume
$$
|\tilde{\na}^k \tilde{\bA}| \le C(n,m,\ep, \psi, |D\psi|, \cdots, |D^{m+1}\psi|) \quad \textrm{for $k \le m-1$} \, .
$$
Note that one has
 \begin{align*}
     \fr{\pa}{\pa \tau} |\tilde{\nabla}^m \tilde{\bA}|^2
     &=
     2 \langle \tilde{\na}^m \bA, \fr{\pa}{\pa \tau} \tilde{\nabla}^m \tilde{\bA} \rangle + 2\psi \tilde{R}^{1/2} | \tilde{\nabla}^m \tilde{\bA} |^2_{\tilde{h}} \, ,
     \\
      \tilde{\dot{F}}^{k \, l} \tna_k \tna_l |\tna^m \tilde{\bA}|^2 &=  2 \langle \tilde{\dot{F}}^{k \, l} \tna_k \tna_l \tna^m \tilde{\bA} , \tna^m \tilde{\bA} \rangle + 2 \psi \tilde{R}^{-1/2} |\tna^{m+1} \tilde{\bA}|^2_{\tilde{H} -\tilde{g} \tilde{h}} \,  .
 \end{align*}
Using these and Lemma \ref{A_m_deriv}, one can compute
 \begin{equation*}%\label{mderiv}
    \begin{split}
    \fr{\pa}{\pa \tau} |\tilde{\nabla}^m \tilde{\bA}|^2
    &\le
    \tilde{\dot{F}}^{k \, l} \tilde{\nabla}_k \tilde{\nabla}_l  |\tilde{\nabla}^m \tilde{\bA}|^2 -2\psi \tilde{R}^{-1/2} |\tilde{\na}^{m+1} \tilde{\bA}|^2_{\tilde{H} \tilde{g}-\tilde{h}} + C_1(n,\ep, \psi, |D \psi|) |\tilde{\na}^{m+1} \tilde{\bA}|^2
    \\
     &+ C_2(n, m, \ep) \sum_{j=2}^m \fr{1}{\tilde{H}^j} \sum_{i_1+\cdots +i_j+k=j} |\tilde{\na}^k \psi | \,  | \tilde{\na}^{m-j+2} \tilde{\bA} |
 \, | \tilde{\na}^{i_1} \tilde{\bA} | \cdots | \tilde{\na}^{i_j} \tilde{\bA}| \,  | \tilde{\na}^m \tilde{\bA} |
     \\
    &+ C_3(n, \ep)\sum_{i+j+k+l=m} |\tilde{\na}^i \psi| \, |\tilde{\na}^j \tilde{\bA}| \, |\tilde{\na}^k \tilde{\bA}| \, |\tilde{\na}^l \tilde{\bA}|  \, |\tilde{\na}^m \tilde{\bA}|
\\
&+ C_4(n,\ep)\sum_{i+j=m} |\tilde{\na}^{i+2} \psi| \, |\tilde{\na}^j \tilde{\bA}| \, |\tilde{\na}^m \tilde{\bA}| \, .
    \end{split}
   \end{equation*}
Using Young's inequality, this yields
  \begin{align*}
     \fr{\pa}{\pa \tau} |\tilde{\nabla}^m \tilde{\bA}|^2
    &\le
    \tilde{\dot{F}}^{k \, l} \tilde{\nabla}_k \tilde{\nabla}_l  |\tilde{\nabla}^m \tilde{\bA}|^2 -\psi_{0} \fr{\ep^{1/2}}{n} |\tilde{\na}^{m+1} \tilde{\bA}|^2
    +C_m (|\tilde{\na}^m \tilde{\bA} |^2+1) \, ,
  \end{align*}
where $C_m =C(n,m,\ep, \tilde{H}_0, \psi, \cdots, |D^{m+2} \psi| )$ is a positive constant.
Consider
\begin{align}\label{A_m_deriv_short}
 \begin{split}
   \fr{\pa}{\pa \tau} [ |\tilde{\nabla}^m \tilde{\bA}|^2 + \tilde{B} | \tna^{m-1} \bA |^2 ]
    &\le
    \tilde{\dot{F}}^{k \, l} \tilde{\nabla}_k \tilde{\nabla}_l  [ |\tilde{\nabla}^m \tilde{\bA}|^2 + \tilde{B} |\tilde{\nabla}^{m-1} \tilde{\bA}|^2 ] +\big(C_m- \tilde{B} \psi_0 \fr{\ep^{1/2}}{n} \big) |\tilde{\nabla}^m \tilde{\bA}|^2
    \\
    &-\psi_{0} \fr{\ep^{1/2}}{n} |\tilde{\na}^{m+1} \tilde{\bA}|^2
     +\tilde{B}C_{m-1}|\tilde{\na}^{m-1}\tilde{\bA} |^2+C_m +\tilde{B}C_{m-1} \,
 \end{split}
\end{align}
where $\tilde{B}$ is a positive constant.  Choose $B$ sufficiently large so that $B \ge \fr{2n C_m}{\ep^{1/2} \psi_0}$, and suppose that $| \tna^{m} \bA |$ is unbounded.
Then this leads to a contradiction when one applies the maximum principle to \eqref{A_m_deriv_short}.
\end{proof}

\section{Global boundedness}\label{global_boundedness}

In this section we show that the curvature quantity defined below is uniformly bounded with the pinching estimate given.  It shall be seen in Section \ref{sec-asymptotic} that the limit manifold under the parabolic rescaling is a round sphere.
Suppose that $h_{ij} \ge \epsilon (H+c)g_{ij}$, for some $\epsilon>0$ and $c>0$, holds initially.  Then from Theorem \ref{thm-pinching}, it remains so until the maximal time $T$.  This implies that $\epsilon H^2 \le R \le H^2$ which will be used repeatedly throughout this section.
Let
$$f_{\sigma}=\frac{|A|^2-H^2/n}{H^{2-\sigma}}$$ and $f=f_{0}$.  
Note that $f_{\sigma} \le O(H^{\sigma})$. 
The aim in this section is to show that $f_{\sigma}$ is bounded for some small $\sigma$.  
With the straightforward computation
 \begin{align*}
& \frac{H^{1+\sigma}}{2R}|\nabla f + \frac{2R}{\psi H^2}\nabla \psi|^2_{|A|^2g-Hh}
\\
&= \frac{H^{1-\sigma}}{2R}|\nabla f_{\sigma}|^2_{|A|^2g-Hh} -\sigma \frac{f_{\sigma}}{H^{\sigma}R} \langle \nabla H, \nabla f_{\sigma} \rangle_{|A|^2g-Hh} +\sigma^2 \frac{f_{\sigma}^2}{2H^{1+\sigma}R}|\nabla H|^2_{|A|^2g-Hh}
\\
&+ \frac{2}{\psi H^{1-\sigma}}\langle \nabla f, \nabla \psi \rangle_{|A|^2g-Hh} +\frac{2R}{\psi^2 H^{3-\sigma}}|\nabla \psi|^2_{|A|^2g-Hh} \,  ,
\end{align*}
the time derivative of $f_{\sigma}$ follows from Lemma \ref{ev-h1}:
 \begin{align}\label{f_sigma1}
  \frac{\partial f_{\sigma}}{\partial t} &= \psi R^{-1/2} \Big( \Box f_{\sigma} -\frac{2R}{H^{5-\sigma}}|H\nabla h_{kl} -h_{kl}\nabla H|^2 -\sigma \frac{H^3 f_{\sigma}}{4R}|\nabla f + \frac{2R}{\psi H^2} \nabla \psi|^2  \nonumber
  \\
  &- \sigma \frac{f_{\sigma}}{H^3}|H \nabla_i h_{kl}-h_{kl}\nabla_i H|^2+\frac{2(1-\sigma)}{H} \langle \nabla H, \nabla f_{\sigma} \rangle_{Hg-h} +\sigma(\sigma-1)\frac{f_{\sigma}}{H^2} |\nabla H|^2_{Hg-h}  \nonumber
  \\
  &+ \frac{H^{1-\sigma}}{2R}|\nabla f_{\sigma}|^2_{|A|^2g-Hh} -\sigma \frac{f_{\sigma}}{H^{\sigma}R} \langle \nabla H, \nabla f_{\sigma} \rangle_{|A|^2g-Hh} +\sigma^2 \frac{f_{\sigma}^2}{2H^{1+\sigma}R}|\nabla H|^2_{|A|^2g-Hh} \nonumber
  \\
  &+ \sigma f_{\sigma} (H|A|^2-tr A^3) \Big) +\frac{2}{R^{1/2}H} \langle \nabla f_{\sigma}-\sigma \frac{f_{\sigma}}{H}\nabla H, \nabla \psi \rangle_{|A|^2g-Hh} 
  \\
  &- \frac{2R^{1/2}}{H^{4-\sigma}} \{ 2\langle \nabla H, \nabla \psi \rangle_{|A|^2g-Hh} -\sigma f_{\sigma} H^{2-\sigma} \langle \nabla H, \nabla \psi \rangle \}+ \sigma \frac{R^{1/2}f_{\sigma}}{\psi H}|\nabla \psi|^2 \nonumber
  \\
  &+ \frac{R^{1/2}}{H^{3-\sigma}} [\sigma(|A|^2-\fr{H^2}{n}) \Delta \psi - 2(|A|^2 g^{ij} - Hh^{ij}) \nabla_i \nabla_j \psi ] \, . \nonumber
   \end{align}
 To estimate the last line above, note that
 \begin{align*}
 \begin{split}
   |\D_i \D_j \psi |^2  &= |  (D^2 \psi)_{ij} - (D_{\bnu} \psi) h_{ij} |^2 \le 2(1+H)^2( || D^2 \psi || + |D\psi| )^2 \, ,
   \\
     | (|A|^2 g^{ij} - Hh^{ij}) \D_i \D_j \psi |^2 &\le   n|A|^2 H^{2-\sigma} f_{\sigma} |\D_i \D_j \psi |^2
   \le 
   n H^{4-\sigma} f_{\sigma} |\D_i \D_j \psi |^2 .
   \end{split}
 \end{align*}
 Thus one has
 \begin{align}\label{psi2nd}
 \begin{split}
  & \frac{R^{1/2}}{H^{3-\sigma}} [\sigma(|A|^2-\fr{H^2}{n}) \Delta \psi - 2(|A|^2 g^{ij} - Hh^{ij}) \nabla_i \nabla_j \psi ]
  \\
  &\le
R^{-1/2} H  (|D\psi| + || D^2 \psi|| )[2n^{1/2} H^{\fr{\sigma}{2}} (H+1) \fs^{1/2} + 2\sigma (H+1) \fs  ]
\end{split}
 \end{align}
 since $|A|^2-\fr{H^2}{n} = H^{2-\sigma} \fs$, $R \le H^2$ and 
  $$
 \nabla_i \nabla_j \psi = D^2 \psi \big(\fr{\pa \bX}{\pa x^i} , \fr{\pa \bX}{\pa x^j}\big) -(\bnu \cdot D\psi) h_{ij} \, .
 $$
Choose $\sigma$ sufficiently small so that
 \begin{align*}
  \sigma(\sigma-1)\frac{f_{\sigma}}{H^2} |\nabla H|^2_{Hg-h} +\sigma^2 \frac{f_{\sigma}^2}{2H^{1+\sigma}R}|\nabla H|^2_{|A|^2g-Hh} &\le 0 .
 \end{align*}
The pointwise bounds for some of the terms in \eqref{f_sigma1} can be easily obtained:
  \begin{align}\label{bound1}
 \begin{split}
 &-\frac{2R^{1/2}}{H^{4-\sigma}} \{ 2\langle \nabla H, \nabla \psi \rangle_{|A|^2g-Hh} -\sigma f_{\sigma} H^{2-\sigma} \langle \nabla H, \nabla \psi \rangle \}
\\
 & + \frac{2}{R^{1/2}H} \langle \nabla f_{\sigma}-\sigma \frac{f_{\sigma}}{H}\nabla H, \nabla \psi \rangle_{|A|^2g-Hh}
   + \sigma \frac{R^{1/2}f_{\sigma}}{\psi H}|\nabla \psi|^2
   \\
 &\le 5H^{\sigma-1} |\nabla H| |\nabla \psi|  
 +
   2\epsilon^{-1/2}(|\nabla f_{\sigma}| + \sigma H^{\sigma-1} |\nabla H| ) |\nabla \psi|
   +
 \sigma H^{\sigma} \frac{|\nabla \psi|^2}{\psi}
 \\
 &\le
10 \epsilon^{-1/2} ( |\nabla f_{\sigma}| + H^{\sigma-1}|\nabla H| ) |\nabla \psi| + \sigma H^{\sigma} \frac{|\nabla \psi|^2}{\psi} .
 \end{split}
\end{align}
Thus the terms involving the gradient of $\psi$ in \eqref{f_sigma1} are bounded above by
 \begin{align}\label{order_remainder}
  \tilde{C}(\epsilon, \sigma, \psi) (|\nabla f_{\sigma}| + H^{\sigma-1}|\nabla H| + H^{\sigma} ),
 \end{align}
since  $f_{\sigma} \le H^{\sigma}$,
where $\tilde{C}(\epsilon, \sigma, \psi) := 10 \epsilon^{-1/2}|\nabla \psi| + \sigma \frac{|\nabla \psi|^2}{\psi} $ for $\epsilon < 1$.
As computed in Lemma 2.3 (ii) in \cite{Hu1}, one can show that
 \begin{align}\label{nabla_H}
  |H\nabla_i h_{kl}-\nabla_i H \cdot h_{kl}|^2
  &\ge
  \frac{1}{2}h_{22}^2 |\nabla H|^2
  \ge
  \frac{1}{2}\epsilon^2 (H+c)^2|\nabla H|^2,
 \end{align}
where the second inequality is obtained by choosing an orthonormal frame with the first element being $\nabla H / |\nabla H|$.
From \eqref{f_sigma1},  \eqref{psi2nd}, \eqref{bound1}, \eqref{order_remainder} and \eqref{nabla_H}, one has
 \begin{align}\label{f_sigma_time_pre}
 \begin{split}
  \frac{\partial f_{\sigma}}{\partial t} &\le \psi R^{-1/2}\Big\{ \Box f_{\sigma} -\epsilon^2 \frac{R}{2H^{3-\sigma}}|\nabla H|^2
  + \frac{2(1-\sigma)}{H} \langle \nabla H, \nabla f_{\sigma} \rangle_{Hg-h}
  \\
  &+\frac{H^{1-\sigma}}{2R}|\nabla f_{\sigma}|^2_{|A|^2g-Hh} -\sigma \frac{f_{\sigma}}{H^{\sigma}R} \langle \nabla H, \nabla f_{\sigma} \rangle_{|A|^2g-Hh} + 
\sigma f_{\sigma} H^3 \Big\} 
  \\
&+ R^{-1/2} H  (|D\psi| + || D^2 \psi|| ) \Big\{  2n^{1/2} H^{\fr{\sigma}{2}} (H+1) \fs^{1/2} + 2 \sigma (H+1) \fs \Big\}
  \\
  &+ \tilde{C}(\epsilon, \sigma, \psi) (|\nabla f_{\sigma}| + H^{\sigma-1}|\nabla H| + H^{\sigma}) \, . 
   \end{split}
 \end{align}

\subsection{$L^p$ bound}\label{lp}

In order to prove that there exists a positive constant $C$ such that
% \begin{align*}
 $ ||f_{\sigma}||_p \le C $
% \end{align*}
for some large $p$, we generalise the argument in Lemma 5.5 in \cite{Hu1}:  it is sufficient to show that 
 \begin{align*}
   \frac{\partial}{\partial t}\int_M f_{\sigma}^p d\mu &\le 0 .
 \end{align*}
Then, multiplying the factor $pf_{\sigma}^{p-1}$ in \eqref{f_sigma_time_pre} and integrating by parts,
it follows that
 \begin{align}\label{f_sigma_1p}
   \frac{\partial}{\partial t} \int_M f_{\sigma}^p d\mu
  &\le
  \int_M \Big[ \psi R^{-1/2}\Big\{ -p(p-1)f_{\sigma}^{p-2}|\nabla f_{\sigma}|^2_{Hg-h}
  -\frac{p}{2} R^{-1} H^{2-\sigma} f_{\sigma}^{p-1}|\nabla f_{\sigma}|^2_{Hg-h}
  \nonumber
  \\
  &+ \sigma\frac{p}{2} R^{-1} H^{1-\sigma} f_{\sigma}^p \langle \nabla H, \nabla f_{\sigma} \rangle_{Hg-h}
  + p H^{-1} f_{\sigma}^{p-1} \langle \nabla H, \nabla f_{\sigma} \rangle_{Hg-h} 
  \nonumber
  \\
  &- p \frac{\epsilon^2}{2} f_{\sigma}^{p-1} R H^{\sigma-3} |\nabla H|^2 
  + p f_{\sigma}^{p-1} \frac{2(1-\sigma)}{H} \langle \nabla H, \nabla f_{\sigma} \rangle_{Hg-h} 
  \\
  &+ p f_{\sigma}^{p-1} \frac{H^{1-\sigma}}{2R}|\nabla f_{\sigma}|^2_{|A|^2g-Hh}
  - p \frac{\sigma}{RH^{\sigma}} f_{\sigma}^{p} \langle \nabla H, \nabla f_{\sigma} \rangle_{|A|^2g-Hh}
  + p \sigma f_{\sigma}^p H^3 \Big\} 
  \nonumber
  \\
  &+ p \fs^{p-1} R^{-1/2} H  (|D\psi| + || D^2 \psi|| ) \Big\{ 2n^{1/2} H^{\fr{\sigma}{2}} (H+1) \fs^{1/2} + 2 \sigma (H+1) \fs \Big\}
  \nonumber
  \\
  &+ p \tilde{C}(\epsilon, \sigma, \psi)  f_{\sigma}^{p-1} (|\nabla f_{\sigma}| +H^{\sigma-1}|\nabla H| + H^{\sigma}) \Big] d\mu .
  \nonumber
 \end{align}
Note that it suffices to consider the case in which
 \begin{align}\label{intfs}
   \int_M f_{\sigma}^p d\mu  > \int_M d\mu \, .
 \end{align}
For $H \ge H_0 := c n \ep $ and $ 0\le k \le 2$,
\begin{align}\label{int1}
 \begin{split}
   \int H^k \fs^p &\le H_0^{k-2} \int H^2 \fs^p d\mu 
   \le
    H_0^{k-2} \int H^2 \fs^p d\mu .
       \end{split}
\end{align}
Similarly, for $ 0\le k \le 3/2$ and $p \ge 2$,
one has
\begin{align}\label{int2}
 \begin{split}
   \int H^k \fs^{p-1/2} 
   \le  \Big( \int H^{\fr{2kp}{2p-1}} \fs^{p}  d\mu \Big)^{1-\fr{1}{2p}}  \Big( \int d\mu\Big)^{1/2p}
   \le  H_0^{k-2}  \int H^2 \fs^p d\mu .
     \end{split}
\end{align}
Then from \eqref{int1}, \eqref{int2} and $R^{-1/2} H \le \ep^{-1/2}$ which follows from the pinching estimate, one has
\begin{align}\label{psideriv}
\begin{split}
&\int_M p \fs^{p-1} R^{-1/2} H  (|D\psi| + || D^2 \psi|| )[2n^{1/2} H^{\fr{\sigma}{2}} (H+1) \fs^{1/2} + 2 \sigma (H+1) \fs ] d\mu 
\\
&\le
2 p \ep^{-1/2} \Big( \sup_M  \Big( \fr{|D\psi|}{\psi} + \fr{|| D^2 \psi||}{\psi} \Big) \Big)( n H_0^{\fr{\sigma}{2}} + \sigma ) (H_0+1) H_0^{-2} 
 \int_M \psi  H^2 \fs^p d\mu .
 \end{split}
\end{align}
Since we want the integral in the right side of \eqref{psideriv} to be bounded above by
\begin{align*}
 \sigma p  \ep^{-1/2} \int_M \psi H^2 f_{\sigma}^{p} d\mu ,
\end{align*}
we take $c$ and $\sigma$ satisfying %\eqref{c_condition2}
\begin{align}\label{c_lp}
 \sup_M  \Big( \fr{|| D\psi ||}{\psi} + \fr{|| D^2 \psi||}{\psi} \Big) \le \fr{(c n \ep)^2 \sigma}{2 [ n (c n \ep)^{\sigma/2} + \sigma ] ( c n \ep +1)}
\end{align}
which holds if one take $\sigma$ sufficiently small, say $o (\ep^{5})$, and $H_0 \ge 1$ satisfying
\begin{align}\label{c_lp2}
\fr{10}{\ep^6}  \sup_M  \Big( \fr{|D\psi|}{\psi} + \fr{|| D^2 \psi||}{\psi} \Big)  \le c .
\end{align}

For the four terms involving $\langle \nabla H, \nabla f_{\sigma} \rangle_{Hg-h}$, noting that $R \ge \epsilon H^2$, $f_{\sigma} \le H^{\sigma}$ and $|A|^2 g-Hh \le H(Hg-h)$, there exists a positive constant $C(\epsilon, \sigma)$ such that those terms are bounded by
 \begin{align}\label{nabla_H_bound}
  p C(\epsilon, \sigma) \int_M \psi \frac{f_{\sigma}^{p-1}}{H}|\nabla H| |\nabla f_{\sigma}| d\mu,
 \end{align}
where $C(\epsilon, \sigma) =O(\epsilon^{-1/2}) $. 
Similarly the three terms involving $|\nabla f_{\sigma}|^2_{Hg-h}$ in \eqref{f_sigma_1p} are bounded by
 \begin{align}\label{nabla_f_sigma1}
  -p(p-1) \int_M \psi R^{-1/2} f_{\sigma}^{p-2}|\nabla f_{\sigma}|^2_{Hg-h} d\mu.
 \end{align}
Thus, from \eqref{f_sigma_1p}, \eqref{psideriv} , \eqref{nabla_H_bound} and \eqref{nabla_f_sigma1},
 \begin{align*}%\label{f_sigma1_bound1}
 \begin{split}
   \frac{\partial}{\partial t} \int_M f_{\sigma}^p d\mu
  &\le -p(p-1) \int_M \psi R^{-1/2} f_{\sigma}^{p-2} |\nabla f_{\sigma}|^2_{Hg-h} d\mu + 2 \sigma p \ep^{-1/2} \int_M \psi H^2 f_{\sigma}^{p} d\mu 
  \\
  &+ p C(\epsilon, \sigma) \int_M \psi \frac{f_{\sigma}^{p-1}}{H} |\nabla H| |\nabla f_{\sigma}| d\mu
   - p \frac{\epsilon^2}{2} \int_M \psi f_{\sigma}^{p-1} R^{1/2} H^{\sigma-3} |\nabla H|^2 d\mu 
  \\
  &+ p \tilde{C}(\epsilon, \sigma, \psi) \int_M f_{\sigma}^{p-1} (|\nabla f_{\sigma}| +H^{\sigma-1} |\nabla H| + H^{\sigma} ) d\mu. 
  \end{split}
 \end{align*}
 Note that \eqref{principal_curv_pinching} implies that $Hg_{ij}-h_{ij} \ge \epsilon Hg_{ij}$ and $\epsilon H^2 \le R \le H^2$.  Then by choosing $p \ge p_0$, where $p_0:=2C(\epsilon,\sigma)^2 \epsilon^{-7/2}+1$, one has
 \begin{align}\label{int_f_sigma}
&  \frac{\partial}{\partial t} \int_M f_{\sigma}^p d\mu
  \le
  -\frac{\epsilon}{2}p(p-1) \int_M \psi f_{\sigma}^{p-2}|\nabla f_{\sigma}|^2 d\mu -\frac{\epsilon^{5/2}}{4}p \int_M \psi f_{\sigma}^{p-1} H^{\sigma-2} |\nabla H|^2 d\mu
  \\
  &+ \frac{2\sigma p}{\epsilon^{1/2}} \int_M \psi H^2 f_{\sigma}^{p} d\mu + p \tilde{C}(\epsilon, \sigma, \psi) \int_M f_{\sigma}^{p-1} (|\nabla f_{\sigma}| + H^{\sigma-1}|\nabla H| + H^{\sigma} ) d\mu. \nonumber
 \end{align}
Note that using Young's inequality, one has
 \begin{align*}
   p H^{\sigma} f_{\sigma}^{p-1}  
       &\le 
       (p-1)H_0^{\sigma_1} H^2 f_{\sigma}^p +  H_0^{-\sigma p} \,  ,
  \end{align*}
where $\sigma_1 = \fr{2p(\sigma-1)+2}{p-1} < 0$. Thus, for $H_0 \ge 1$ and sufficiently large $p$, one has
 \begin{align*}
  p \int_M H^{\sigma} f_{\sigma}^{p-1}  d\mu
  &
  \le \frac{p-1}{\psi_0} H_0^{\sigma_1} \int_M \psi H^{2} f_{\sigma}^{p} d\mu + \int_M H_0^{-\sigma p} d\mu
  \le  \frac{2p}{\psi_0} H_0^{\sigma_1} \int_M \psi H^{2} f_{\sigma}^{p} d\mu \, ,
  \end{align*}
where  $\psi_0:=\inf_{\R^{n+1}} \psi$, and similarly,
 \begin{align*}
 p \int_M f_{\sigma}^{p-1} H^{\sigma-1}|\nabla H| d\mu
 &\le
  \frac{\beta p^2}{\psi_0} \int_M \psi H_0^{\sigma-2} f_{\sigma}^{p-1} |\nabla H|^2 d\mu
 + \frac{2H_0^{\sigma_1}}{\beta \psi_0} \int_M \psi H^2 f_{\sigma}^{p} d\mu, \nonumber
\\
  p \int_M f_{\sigma}^{p-1} |\nabla f_{\sigma}| d\mu
  &\le
  \frac{\beta p^2}{\psi_0} \int_M \psi f_{\sigma}^{p-2} |\nabla f_{\sigma}|^2 d\mu \nonumber
  + \frac{2 H_0^{\sigma_1}}{\beta \psi_0} \int_M \psi H^2 f_{\sigma}^{p} d\mu, \nonumber
 \end{align*}
where $\beta$ is a positive constant. 
For $H_0 \ge 1$, choose $\tilde{C}(\epsilon, \sigma, \psi)$ satisfying $\tilde{C}(\epsilon, \sigma, \psi) < \min \{ \ep^{5/2},  \sigma \ep^{-1/2} H_0^{-\sigma_1} \}$ which holds if 
$\tilde{C}(\epsilon, \sigma, \psi) < \ep^{9/2}$.
Thus, by choosing $\beta= \fr{1}{p}$, \eqref{int_f_sigma} yields
 \begin{align}\label{f_sigma_1_Lp2}
 \begin{split}
  \frac{\partial}{\partial t}\int_M f_{\sigma}^p d\mu
  &
  \le
  -\frac{\epsilon}{4}p(p-1) \int_M \psi f_{\sigma}^{p-2} |\nabla f_{\sigma}|^2 d\mu
  \\
  &
-\frac{\epsilon^{5/2}}{8} p \int_M \psi H^{\sigma-2} f_{\sigma}^{p-1} |\nabla H|^2 d\mu
  + \frac{4\sigma p}{\epsilon^{1/2}} \int_M \psi H^2 f_{\sigma}^{p} d\mu. 
  \end{split}
 \end{align}
 To eliminate the last integral above, we apply the following Michael-Simon Sobolev type inequality as given in Lemma 5.4 in \cite{Hu1}.
  \begin{lemma}\label{huisken1}
If $H>0$ and $h_{ij} \ge \epsilon (H+c)g_{ij}$ for some $\epsilon>0$ and $c>0$ initially, then we have
    \begin{align*}
      \int_M \psi f_{\sigma}^{p} H^2 d\mu
      &\le
      \frac{(2\gamma p+5)\psi_1}{n \epsilon^2 \psi_0} \int_M \psi \frac{f_{\sigma}^{p-1}}{H^{2-\sigma}} |\nabla H|^2 d\mu + \frac{(p-1) \psi_1}{n\epsilon^2 \gamma \psi_0} \int_M \psi f_{\sigma}^{p-2}|\nabla f_{\sigma}|^2 d\mu,
    \end{align*}
  for $p \ge 2$, any $\gamma>0$ and any $0 \le \sigma \le 1/2$, where $\psi_1:=\sup_{\R^{n+1}}\psi$.
  \end{lemma}
  From Lemma \ref{huisken1} and \eqref{f_sigma_1_Lp2}, one has
 \begin{align}\label{time_f_sigma_p}
  \frac{\partial}{\partial t}\int_M f_{\sigma}^p d\mu
  &
  \le
  \Big( \frac{  4\psi_{1} \sigma p}{n\gamma \epsilon^{5/2} \psi_{0}}(p-1) -\frac{\epsilon}{4}p(p-1) \Big) \int_M \psi f_{\sigma}^{p-2} |\nabla f_{\sigma}|^2 d\mu
  \\
  &+ \Big( \frac{ 4 \psi_{1} \sigma p}{n\epsilon^{5/2} \psi_{0}}(2\gamma p+5)- \frac{\epsilon^{5/2}}{8} p \Big) \int_M \psi H^{\sigma-2} f_{\sigma}^{p-1} |\nabla H|^2 d\mu. \nonumber
 \end{align}
By choosing $\gamma = \frac{8\psi_1 \sigma}{n \epsilon^{7/2}\psi_0}$ and $\sigma \le o(\ep^5)$, one concludes that  for $c$ satisfying \eqref{c_lp2},
 \begin{align}\label{p_sigma_relation}
  \frac{\partial}{\partial t}\int_M f_{\sigma}^p d\mu &\le 0 ,   \qquad \textrm{for} \quad p_0 \le p \le p_{\sigma}:=\frac{n \epsilon^{7/2} \psi_0}{32 \sigma \psi_1} \Big(\frac{n\epsilon^5}{ 32 \psi_1 \sigma}-5 \Big) .
 \end{align}
Therefore one can conclude that for $c$ satisfying \eqref{c_condition} and \eqref{c_lp2},
\begin{align}\label{f_sigma_Lp_uniform_bound}
  ||f_{\sigma}||_p &\le C   \qquad \textrm{for  $p_0 \le p \le p_{\sigma}$.}
 \end{align}

\subsection{Moser iteration}

In this subsection, we obtain the uniform bound for $\fs$ for some small $\sigma$ assuming that  $c$ in the pinching estimate in \eqref{c_condition} satisfies \eqref{c_lp2}.
Let $\eta$ be a smooth test function which will be given explicitly later.
Integrating by parts, one has
\begin{align*}%\label{box_fsp_eta}
   & \int_M \eta^2 \psi R^{-1/2} \Box f_{\sigma}^p d\mu =
   \\
   &-\frac{p}{2} \int_M \eta^2 \psi R^{-3/2} H^{2-\sigma} f_{\sigma}^{p-1}|\nabla f_{\sigma}|^2_{Hg-h}
    + \sigma \frac{p}{2} \int_M \eta^2 \psi R^{-3/2} H^{1-\sigma} f_{\sigma}^p \langle \nabla H, \nabla f_{\sigma} \rangle_{Hg-h} \nonumber
    \\
   &+ p \int_M \eta^2 \psi R^{-1/2} H^{-1} f_{\sigma}^{p-1} \langle \nabla H, \nabla f_{\sigma} \rangle_{Hg-h}
    -p \int_M R^{-1/2} f_{\sigma}^{p-1} \langle \nabla (\eta^2 \psi), \nabla f_{\sigma} \rangle_{Hg-h}. \nonumber
 \end{align*}
 Note that only the last term above involves the derivative of $\eta$.
Multiplying \eqref{f_sigma_time_pre} by $p \eta^2 f_{\sigma}^{p-1}$ and proceeding as in subsection \ref{f_sigma_time_pre}, from \eqref{f_sigma_1p}, using the pinching estimate and Young's inequality, one then obtains
 \begin{align}\label{f_sigma_1_Lp2_eta}
  & \frac{\partial}{\partial t}\int_M \eta^2 f_{\sigma}^p d\mu  -2\int_M f_{\sigma}^p \eta \frac{\partial \eta}{\partial t} d\mu  
  \nonumber
  \\
  &\le -\frac{\epsilon}{4}p(p-1) \int_M  \eta^2 \psi f_{\sigma}^{p-2} |\nabla f_{\sigma}|^2 d\mu
  -\frac{\epsilon^{5/2}}{8} p \int_M \eta^2 \psi H^{\sigma-2} f_{\sigma}^{p-1} |\nabla H|^2 d\mu
  \nonumber
  \\
  &+\frac{2\sigma p}{\epsilon^{1/2}} \int_M \eta^2 \psi H^2 f_{\sigma}^{p} d\mu 
  + \fr{2}{\epsilon^{1/2}} \int_M \psi \eta |\nabla \eta| |\nabla f_{\sigma}^p| d\mu
     - \epsilon^{\frac{1}{2}} \int_M \eta ^2 \psi H^2 f_{\sigma}^p d\mu
     \\
   &+
   2p \Big( \sup_M  \Big( \fr{|D\psi|}{\psi} + \fr{|| D^2 \psi||}{\psi} \Big) \Big)  \int_M \psi \eta^2 [ \sigma (H+1) \fs^p  + 2n^{1/2} H^{\sigma/2} (H+1) \fs^{p-1/2} ] d\mu  \nonumber . 
 \end{align}
 The last integral in the third line comes from the derivative of the measure $d\mu$ and that $R \ge \epsilon H^2$.
Let $\epsilon_i$, $i=1,2,3$, be any positive numbers.  
Since
 \begin{align*} %\label{useful_ineq}
\eta^2 |\nabla f_{\sigma}^{\frac{p}{2}}|^2  &=  |\nabla(\eta f_{\sigma}^{\frac{p}{2}})|^2 +f_{\sigma}^p |\nabla \eta|^2-2 f_{\sigma}^{\frac{p}{2}}\langle \nabla (\eta f_{\sigma}^{\frac{p}{2}}), \nabla \eta \rangle,
 \\
2 f_{\sigma}^{\frac{p}{2}} |\langle \nabla (\eta f_{\sigma}^{\frac{p}{2}}), \nabla \eta \rangle| &\le \epsilon_1 |\nabla (\eta f_{\sigma}^{\frac{p}{2}})|^2 + \epsilon_1^{-1} f_{\sigma}^p |\nabla \eta|^2,
\\
\eta \psi |\nabla \eta| |\nabla f_{\sigma}^p|
  &\le
   2\epsilon_2 \psi |\nabla (\eta f_{\sigma}^{\frac{p}{2}})|^2 +  2\epsilon_2^{-1} \psi f_{\sigma}^p |\nabla \eta|^2 ,
 \end{align*}
   taking $\epsilon_1=\frac{1}{2}$ and $\epsilon_2=\frac{1}{32}\epsilon^{3/2}$, \eqref{f_sigma_1_Lp2_eta} becomes
 \begin{align}\label{f_sigma_1_Lp2_eta3}
 & \frac{\partial}{\partial t}\int_M \eta^2 f_{\sigma}^p d\mu   -2\int_M f_{\sigma}^p \eta \frac{\partial \eta}{\partial t} d\mu
 \nonumber
 \\
  &\le
  -\frac{\epsilon}{8} \int_M  \psi |\nabla (\eta f_{\sigma}^{\frac{p}{2}})|^2 d\mu +\frac{200}{\epsilon^2} \int_M \psi |\nabla \eta|^2 f_{\sigma}^{p} d\mu
  \nonumber
  \\
  &-\frac{\epsilon^{5/2}}{8} p \int_M \eta^2 \psi H^{\sigma-2} f_{\sigma}^{p-1} |\nabla H|^2 d\mu 
   +\frac{2\sigma p}{\epsilon^{1/2}} \int_M \eta^2 \psi H^2 f_{\sigma}^{p} d\mu
   \\
   &- \epsilon^{\frac{1}{2}} \int_M \eta ^2 \psi H^2 f_{\sigma}^p d\mu
   \nonumber
   \\
  &+ 2p \Big( \sup_M  \Big( \fr{|D\psi|}{\psi} + \fr{|| D^2 \psi||}{\psi} \Big) \Big)  \int_M \psi \eta^2 [ \sigma (H+1) \fs^p  + 2n^{1/2} H^{\sigma/2} (H+1) \fs^{p-1/2} ] d\mu 
  \nonumber
 \end{align}
With this parabolic equation, we run the Moser iteration which is also useful for extending mean curvature flow past singular time as shown in \cite{LS} and \cite{LXYZ}.
Rescale and translate time $t$ in $[T - \tilde{\delta}, T)$ by $\tilde{\delta}^{-1} (t - T + \tilde{\delta})$, for some $\tilde{\delta} > 0$, so that the rescaled time, also denoted by $t$, is in $[0, 1)$.
Let
 \begin{align*}
  D &= \cup_{0 \le t \le 1}(B(x_0,1)\cap M_t), \qquad \tilde{D} = \cup_{\frac{1}{12} \le t \le 1}(B(x_0,\frac{1}{2})\cap M_t),
 \end{align*}
 where $x_0$ is the limit point of $M_t$,
and let
 \begin{align*}
  r_k &= \frac{1}{2}+\frac{1}{2^{k+1}},
   \qquad 
   t_k=\frac{1}{12}(1-\frac{1}{4^k}),
  \qquad 
   \rho_k = r_{k-1}-r_k=\frac{1}{2^{k+1}}.
 \end{align*}
Consider the set
 \begin{align*}
  D_k &=  \cup_{t_k \le t \le 1}(B(x_0, r_k)\cap M_t).
 \end{align*}
 Note that $D_0=D$ and $t_k-t_{k-1}=\rho_k^2$.  For convenience, we write $M$ for $M_t$.
Let $\eta=\eta_k$ be the smooth test function defined on $M \times [0,1)$ by
$$
\eta_k(x,t) := v_k(|x-x_0|^2)\phi_k(t),
$$
where
 \begin{eqnarray}\label{v_def}
   v_k(s) =
   \begin{cases}
     1 & \text{for $s \le r_k^2$},
     \\
     0 & \text{for $s \ge r_{k-1}^2$},
    \end{cases}
     \end{eqnarray}
     and $v_k(s) \in [0,1]$ with $|v_k'(s)| \le c_n \rho_k^{-2}$ for $r_k^2 \le s \le r_{k-1}^2$, and
\begin{eqnarray}\label{phi_def}
   \phi_k(t) =
   \begin{cases}
     0 & \text{for $0 \le t \le t_{k-1}$},
     \\
     1 & \text{for $t_k \le t \le 1$},
    \end{cases}
    \end{eqnarray}
and $\phi_k(t) \in [0,1]$ with $|\phi_k'(t)| \le c_n \rho_k^{-2}$ for $t_{k-1} \le t \le t_k$.
From \eqref{f_sigma_1_Lp2_eta3},  in the time internal $[0, 1)$,
 \begin{align}\label{pre_moser1}
  & \tilde{\delta} \frac{\partial}{\partial t}\int_M \eta^2 f_{\sigma}^p d\mu
  +\frac{\epsilon}{8} \int_M \psi |\nabla (\eta f_{\sigma}^{\frac{p}{2}})|^2 d\mu
  +\epsilon^{\frac{1}{2}} \int_M \eta ^2 \psi H^2 f_{\sigma}^p d\mu
  \nonumber
  \\
  &\le
  \frac{200}{\epsilon^2} \int_M \psi |\nabla \eta|^2 f_{\sigma}^{p} d\mu
    +2 \tilde{\delta} \int_M f_{\sigma}^p \eta \frac{\partial \eta}{\partial t} d\mu +\frac{2\sigma p}{\epsilon^{1/2}} \int_M \eta^2 \psi H^2 f_{\sigma}^{p} d\mu  
      \\
      &+
   2p  \Big( \sup_M  \Big( \fr{|| D\psi ||}{\psi} + \fr{|| D^2 \psi||}{\psi} \Big) \Big)   \int_M \psi \eta^2 [ \sigma (H+1) \fs^p  + 2n^{1/2} H^{\sigma/2} (H+1) \fs^{p-1/2} ] d\mu
   \nonumber
 \end{align}
For $u\in W^{1,1}(M)$ and $T_1 \le T_2$, using the Sobolev inequality and the Schwarz inequality, one has
\begin{align}\label{sobolev1}
 \int_{T_1}^{T_2} \Big( \int_M u^{\frac{2n}{n-1}} d\mu \Big)^{\frac{n-1}{n}}dt &\le c(n) \Big( \sup_{[T_1,T_2)} \int_M u^2 d\mu \Big)^{\frac{1}{2}} \Big( \int_{T_1}^{T_2} \int_M (|\nabla u|^2+H^2 u^2) d\mu dt \Big)^{\frac{1}{2}},
\end{align}
and using the interpolation inequality,
\begin{align*}
 \int_M u^{\frac{2(n+1)}{n}} d\mu &\le \Big( \int_M u^2 d\mu \Big)^{\frac{1}{n}}  \Big( \int_M u^{\frac{2n}{n-1}} d\mu \Big)^{\frac{n-1}{n}},
\end{align*}
one has
\begin{align}\label{sobolev2}
  \int_{T_1}^{T_2} \int_M u^{\frac{2(n+1)}{n}} d\mu dt &\le c(n) \Big( \sup_{[T_1,T_2)} \int_M u^2 d\mu \Big)^{\frac{n+2}{2n}} \Big( \int_{T_1}^{T_2} \int_M (|\nabla u|^2+H^2 u^2) d\mu dt \Big)^{\frac{1}{2}}.
\end{align}
Integrating \eqref{pre_moser1} over $[0,1)$, we have
\begin{align}
 &\tilde{\delta}  \sup_{t \in [0,1)} \int_M \eta^2 f_{\sigma}^p d\mu
  +\frac{\psi_0 \epsilon}{8} \int_{0}^1 \int_{M} |\nabla (\eta f_{\sigma}^{\frac{p}{2}})|^2 d\mu dt
  +\epsilon^{\frac{1}{2}} \psi_0
   \int_{0}^1 \int_{M} \eta ^2 H^2 f_{\sigma}^p d\mu dt
  \\
  &\le
  \frac{200\psi_1 }{\epsilon^2} \int_{0}^1 \int_{M} |\nabla \eta|^2 f_{\sigma}^{p} d\mu dt
    + \tilde{\delta} \int_{0}^1 \int_{M} 2f_{\sigma}^p \eta \frac{\partial \eta}{\partial t} d\mu dt  \nonumber
  + \frac{2 \psi_1 \sigma p}{\epsilon^{1/2}}  \int_{0}^1 \int_{M} \eta ^2 H^2 f_{\sigma}^p d\mu dt   \nonumber
   \\
     &+
   2p C' (\psi) \int_0^1 \int_M \psi \eta^2 [ \sigma (H+1) \fs^p  + 2n^{1/2} H^{\sigma/2} (H+1) \fs^{p-1/2} ] d\mu dt ,  \nonumber
\end{align}
where $C' (\psi) = \sup_M  \Big( \fr{|| D\psi ||}{\psi} + \fr{|| D^2 \psi||}{\psi} \Big)$.
Note that one can choose $\tilde{\delta}$ small enough so that $H \ge 1$ in $[0, 1)$.  This can be achieved applying the pinching estimate in Corollary \ref{eigenvalue_pinching}.
Assuming that $\epsilon \psi_0  \le 8$ and $\tilde{\delta} \le1$, this implies
 \begin{align}\label{pre_moser2}
  &  \iint_{supp \, \eta}  ( |\nabla (\eta f_{\sigma}^{\frac{p}{2}})|^2 + \eta ^2 H^2 f_{\sigma}^p ) d\mu dt \nonumber
 \\
 &\le
 \frac{1600 \psi_1}{\epsilon^3 \psi_0} \iint_{supp \, \eta} f_{\sigma}^p ( |\nabla \eta|^2
    + 2\eta \frac{\partial \eta}{\partial t} + \sigma p \eta^2 H^2 ) d\mu dt ,
 \end{align}
 and the same bound also holds for
$
\tilde{\delta}  \sup_{[0,1)} \int_M \eta^2 f_{\sigma}^p d\mu .
$
For $H \ge 1$, denoting $f_{\sigma}^p H^2 = f_{\tilde{\sigma}}^p$ where $\tilde{\sigma}=\sigma +\frac{2}{p}$, 
one has
$f_{\sigma} \le f_{\tilde{\sigma}}$.
 Substituting $u$ by $\eta f_{\sigma}^{\frac{p}{2}}$ in \eqref{sobolev2}, one obtains
 \begin{align}\label{pre_moser3}
  \iint_{supp \, \eta}  (\eta f_{\sigma}^{\frac{p}{2}})^{\frac{2(n+1)}{n}} d\mu dt
  &\le
  c(n,\epsilon) \Big(   \iint_{supp \, \eta} f_{\sigma}^{p} ( |\nabla \eta|^2
    + 2\eta \frac{\partial \eta}{\partial t} ) +
     (1+C'')  \sigma p \eta^2 f_{\tilde{\sigma}}^p) d\mu dt
   \\
   &+
  p C'' \int_0^1 \big( \int_M \eta^2 f_{\tilde{\sigma}}^p d\mu \big)^{1-\fr{1}{2p}} dt    \nonumber
   \Big)^{\frac{n+1}{n}}
 \end{align}
 where $c(n,\epsilon) := c(n) \Big(\frac{1600 \psi_1}{\epsilon^3 \psi_0} \Big)^{\frac{n+1}{n}}$ and $C''= C'(\psi) \psi_1$.
Typically, as in \cite{E1,LS}, one has
 \begin{align}\label{eta1}
   |\nabla \eta_k|^2 + \frac{\partial}{\partial t} \eta_k^2 &\le \tilde{c}(n) \rho_k^{-2}= \tilde{c}(n) 4^k  \qquad \textrm{ on $D_{k-1}$}
 \end{align}
 and the left hand side vanishes in $M \times [0,1) \backslash D_{k-1}$, where $\tilde{c}(n)$ is a constant depending only on $n$.
 With out loss of generality, suppose that there is some point where $ f_{\tilde{\sigma}} > C$ for some constant $C > 0$ so that
 \begin{align*}
 \begin{split}
  \int_0^1 \big( \int_M \eta^2 f_{\tilde{\sigma}}^p  d\mu \big)^{1-\fr{1}{2p}} dt 
&\le 
\tilde{C}  \iint_{supp \, \eta} \eta^2 f_{\tilde{\sigma}}^p d\mu dt ,
\end{split}
  \end{align*}
for some constant $\tilde{C} > 0$.
Thus for $\tilde{\sigma} \ge \sigma$, \eqref{pre_moser3} and \eqref{eta1} yield
\begin{align}\label{pre_moser4}
  \iint_{supp \, \eta_k}  (\eta_k^2 f_{\sigma}^p)^{\frac{n+1}{n}} d\mu dt
  &\le
  \tilde{c}(n,\epsilon) \Big( \iint_{supp \, \eta_k} 4^k p f_{\tilde{\sigma}}^{p} d\mu dt
   \Big)^{\frac{n+1}{n}}
 \end{align}
where $\tilde{c}(n,\epsilon):= ( C + \sigma )  (1+C'')  c(n,\epsilon) (2\tilde{c}(n))^{\frac{n+1}{n}}$.  
Let $\lambda=\frac{n+1}{n}$, $p=\lambda^{k-1}$ and $\sigma_k = \sigma + 2\lambda^{-k+1}$.
If $I_p(t) \ge 1$, then
\eqref{pre_moser4} implies that
 \begin{align}\label{pre_moser5}
   \iint_{D_k} \eta_k^{2\lambda} f_{\sigma}^{\lambda^k} d\mu dt
 &\le
 \tilde{c}(n,\epsilon)  \Big( \iint_{D_{k-1}} 4^k \lambda^{k-1} f_{\sigma_k}^{\lambda^{k-1}} d\mu dt \Big)^{\lambda} \, ,
 \end{align}
 since $\eta_k \equiv 1$ on $D_k$ and $supp\, \eta_k \subset D_{k-1}$.
 That is,
 \begin{align}\label{pre_moser7}
  ||f_{\sigma}||_{L^{\lambda^k}(D_k)} &\le (\tilde{c}(n,\epsilon) ^{\lambda^{-1}} 4^k \lambda^{k-1})^{\lambda^{-k+1}} ||f_{\sigma_k}||_{L^{\lambda^{k-1}}(D_{k-1})} .
 \end{align}
Note $\sum_{k=1}^{\infty} k \lambda^{-k}=O(1)$, and in \eqref{p_sigma_relation}, $\sigma$ can be chosen sufficiently small so that
 $$
 \sigma-2\sum_{j=0}^{\infty} \lambda^{-k_{\sigma}+1-j} =\sigma-2(n+1)p_{\sigma}^{-1} >0 ,
 $$
 where $p_{\sigma}=\lambda^{k_{\sigma}-1}$ for some $k_{\sigma}>1$ since $p_{\sigma} =O(\epsilon^{17/2} \sigma^{-2})$ for $\sigma \le o(\ep^5)$.
   Thus from \eqref{f_sigma_Lp_uniform_bound}, one has an iteration relation:
 \begin{align}\label{pre_moser8}
  ||f_{\sigma}||_{L^{\infty}(\tilde{D})} &\le c'(n,\epsilon)||f_{\sigma_{k_{\sigma}}}||_{L^{\lambda^{k_{\sigma}-1}}(D_{k_{\sigma}-1})} \le C,
 \end{align}
for a fixed $k_{\sigma} = k_{\sigma} (\ep, \sigma)$ and some constants $c'(n,\epsilon)$
 $=\big( \prod_{k=k_{\sigma}}^{\infty} ( \tilde{c}(n,\ep)^{\lambda^{-1}} 4^k \lambda^{k-1} )^{\lambda^{-k+1}}   \big)^{-1}$ 
and $C > 0$, where the last inequality follows from \eqref{f_sigma_Lp_uniform_bound}.
Therefore,
$$
\sup_{[T-\delta',T)} \sup_{M \cap B(x_0, \frac{1}{2})} f_{\sigma} \le C,
$$
where $\delta'=\frac{\tilde{\delta}}{12}$.
From Lemma \ref{H_upperbound}, we conclude that
\begin{theorem}\label{L_infinity_bound}
If $h_{ij} \ge \epsilon(H+c)g_{ij}$ \, for some $\epsilon>0$ and $c>0$ initially satisfying \eqref{c_condition} and \eqref{c_lp2} , then one has
 \begin{align*}
   |A|^2-\frac{H^2}{n} &\le C H^{2-\tilde{\sigma}}, \qquad \textrm{ for some small $\tilde{\sigma}$}.
 \end{align*}
 \end{theorem}

\section{Proofs of main theorems}\label{sec-asymptotic}

\subsection{Proof of Theorem \ref{thm-main-1}}
From Lemma \ref{lem-iso-w} and the containment principle, one can conclude that $X(\cdot,t)$ converges to a point as $t$ tends to $T$ via the regularity theory of uniformly parabolic equations (see, for example, \cite{KS} for the regularity theory).
\qed

\subsection{Limit equation in dimension two}

The monotone quantity is a useful tool to analyze the asymptotic behavior of geometric flows.  For the mean curvature flow, a monotonicity formula using the backwards heat kernel gives a limit equation which leads to the classification of self-similar solutions \cite{Hu2}.  Here, we simply use the volume of a convex region with its boundary being $M_t$ to normalize the hypersurface.
In general, without the divergence structure for the speed $F$ depending on the curvature, it is difficult to deduce a limit equation.  However, in dimension two, this can be overcome since $R=2K$, where $K$ is the Gauss curvature, and $K$ has a quantity that is not quite monotone but enough to obtain the limit behaviour.  For this reason, we consider $X_t=-\psi (2K)^{1/2} \boldsymbol{\nu}$ which coincides with the flow \eqref{eq-main} in dimension two, and call this the anisotropic Gauss curvature flow (to be precise, (1/2)-Gauss curvature flow).

The (half) volume of a convex region with its boundary $M_t$ can be written in an integral form using the support function $S$:
\begin{equation*}
 \begin{split}
  V(t)&=\frac{1}{n+1}\int_{\mS^n}\frac{S}{2K}\,d\sigma_{\mS^n},
 \end{split}
\end{equation*}
where $d\sigma_{\mS^n}$ is the standard measure on $\mS^n$.
This is used to define a mixed volume of convex regions in \cite{A5} where it is shown that given specific speeds of evolution, some dilation invariant integral quantities monotonically decrease.
\begin{lemma}
Under the flow (\ref{evolution_eq1}), we have
  \begin{align*}
  \frac{\partial}{\partial t} V(t) &= -\int_{\mS^n} \frac{\psi}{(2K)^{1/2}} \,d\sigma_{\mS^n}.
  \end{align*}
\end{lemma}

\begin{proof}
  Denote $\cK = 1/K$.  Using integration by parts and the fact that $\overDi\cK (h^{-1})^{ij}=0$, one has
  \begin{equation*}
          \begin{split}
\int_{\mS^n}S\cK_t\,d\sigma_{\mS^n}
&=\int_{\mS^n}S\cK (h^{-1})^{ij}(\overDi\overDj S_{t}+S_t\overline{g}_{ij})\,d\sigma_{\mS^n}\\
&=\int_{\mS^n}S_{t}\cK (h^{-1})^{ij}(\overDi\overDj S+S\overline{g}_{ij})\,d\sigma_{\mS^n}
=n\int_{\mS^n}S_{t}\cK\,d\sigma_{\mS^n}.
          \end{split}
         \end{equation*}
         Then we have
   \begin{equation*}
         \begin{split}
\frac{\partial}{\partial t} V(t)&=\frac{1}{2(n+1)}\int_{S^n}(\cK S_t+S\cK_t)\,d\sigma_{\mS^n}
=-\int_{S^n} \frac{\psi}{(2K)^{1/2}} \,d\sigma_{\mS^n}.
         \end{split}
         \end{equation*}
\end{proof}
From this, one can write
$V(t)=V(0)-\int_0^t\eta(s)\,ds$ where $\eta(t):= \int_{\mS^n}\frac{\psi}{(2K)^{1/2}} \,d\sigma_{\mS^n}$.
In order to normalize the volume, rescale the hypersurface by
  \begin{align*}
    \hat{X}(\tau) &= \frac{X(t)}{V(t)^{1/(n+1)}}  \qquad \textrm{and}
    \qquad
    \tau (t) = -\log \big( \fr{V(t)}{V(0)} \big).
  \end{align*}
One can easily compute that
\begin{align}\label{eq-gcf-scaled}
 \frac{\partial \hat{X}}{\partial \tau} &= -\frac{\hat{\psi} (2\hat{K})^{1/2}}{\hat{\eta }} \hat{\bnu} + \frac{1}{n+1}\hat{X}.
\end{align}
\begin{lemma}\label{monotone}
Let
\begin{align*}
\hat{\cI}(\tau) &= \left(\int_{\mS^n} \frac{\hat{\psi}^2}{\hat{S}} \,d\sigma_{\mS^n}\right)^{-1}.
\end{align*}
In dimension two, under the volume preserving rescaling, with the initial pinching condition satisfying \eqref{c_condition1} given, one has
    $$
    \frac{d}{d \tau} \hat{\cI}(\tau) \rightarrow 0,
    $$
as $\tau \rightarrow \infty$, and the limit profile satisfies
    $
     \hat{S}^* = C \hat{\psi} (\hat{R}^*)^{1/2}
    $
    for some constant $C>0$, where $\hat{S}^{*}$ and $\hat{R}^*$ are the support function and the scalar curvature of the rescaled limit manifold $\hat{M}^*$, respectively.
\end{lemma}
\begin{proof}
From \eqref{eq-gcf-scaled}, one obtains
\begin{equation*}
\frac{\hat{\psi}^2}{\hat{S}^2} \left(\frac{\partial \hat{S}}{\partial \tau}-\frac{1}{n+1}\hat{S}\right) = -\frac{\hat{\psi}^3 (2\hat{K})^{1/2}}{\hat{\eta }\hat{S}^2} \, ,
\end{equation*}
which implies
 \begin{align*}
\frac{d}{d \tau} \hat{\cI}(\tau)
&=
\hat{\cI}(\tau)^{2} \big[ \frac{1}{n+1} \int_{\mS^n}\frac{\hat{\psi}^2}{\hat{S}}\,d\sigma_{\mS^n}
 -\fr{1}{\hat{\eta}}\int_{\mS^n}\frac{\hat{\psi}^3 (2\hat{K})^{1/2}}{ \hat{S}^2}\,d\sigma_{\mS^n}
 - 2\int \fr{\hat{\psi}}{\hat{S}} \fr{\pa \hat{\psi}}{\pa \tau} \big] \, .
 \end{align*}
Using H\"older's inequality and the definition of $\eta$ yield
\begin{align*}
& \left( \int_{\mS^n}\frac{\hat{\psi}^2}{\hat{S}}\,d\sigma_{\mS^n} \right) \left( \int_{\mS^n}  \frac{\hat{\psi}}{(2\hat{K})^{1/2}} \,d\sigma_{\mS^n} \right)^{1/2}
\\
&\le
\left( \int_{\mS^n}\frac{\hat{\psi}^3 (2\hat{K})^{1/2}}{\hat{S}^{2}}\,d\sigma_{\mS^n} \right)^{1/2}  \left( \int_{\mS^n} \frac{\hat{S}}{2\hat{K}} \, d\sigma_{\mS^n} \right)^{1/2}
\left( \int_{\mS^n} \frac{\hat{\psi}^2}{\hat{S}} \, d\sigma_{\mS^n}  \right)^{1/2}.
\end{align*}
The fact that $\hat{V}(\tau) = \frac{1}{n+1}\int_{\mS^n} \frac{\hat{S}}{2\hat{K}} \, d\sigma_{\mS^n} = 1$ implies
\begin{align*}
\left( \int_{\mS^n}\frac{\hat{\psi}^2}{\hat{S}}\,d\sigma_{\mS^n} \right) \left( \int_{\mS^n}  \frac{\hat{\psi}}{(2\hat{K})^{1/2}} \,d\sigma_{\mS^n} \right)
&\le
(n+1) \int_{\mS^n}\frac{\hat{\psi}^3(2\hat{K})^{1/2}}{\hat{S}^2}\,d\sigma_{\mS^n} \, ,
\end{align*}
where the equality holds if and only if
 $
  \hat{S} = C \hat{\psi} \hat{K}^{1/2}
 $
 for some constant $C>0$, and therefore,
  \begin{align*}
    \frac{d}{d \tau} \hat{\cI}(\tau) &\le -2 \hat{\cI}(\tau)^2 \hat{\eta}\int_{\mS^n} \fr{\hat{\psi}}{\hat{S}} \fr{\pa \hat{\psi}}{\pa \tau} d\sigma_{\mS^n} \, .
  \end{align*}
Since $D\psi \rightarrow 0$ as $\tau \rightarrow \infty$ and, in dimension two, the pinching estimate controls $\hat{K}$ and $\hat{S}$, one has $\fr{\pa \hat{\psi}}{\pa \tau} \rightarrow 0$ as $\tau \rightarrow \infty$.  Also $\hat{V}(\tau)=1$ implies that
 \begin{align*}
  \int_{\mS^2} \fr{\hat{\psi}}{\hat{S}} d\sigma_{\mS^2} \le C,
 \end{align*}
for some positive constant $C$, and hence,
 \begin{align*}
   \lim_{\tau \rightarrow \infty} \frac{d}{d \tau} \hat{\cI}(\tau) &= 0 \, ,
  \end{align*}
so that the limit profile satisfies $\hat{S}^* = C \hat{\psi} (\hat{R}^*)^{1/2} $
 for some constant $C>0$, where $\hat{S}^{*}$ and $\hat{R}^*$ are the support function and the scalar curvature of $\hat{M}^*$, respectively. 
  \end{proof}

\subsection{Proof of Theorem \ref{thm-main-2}}

Parametrizing the rescaled hypersurface as a graph by
  \begin{align}\label{rescaled_X_graph}
     \tilde{X}(x,t) &= \tilde{r}(z,t) z,
  \end{align}
      where $x=\pi^{-1}(z)$, $z \in S^n$ and $\pi: M^n \rightarrow S^n$ is the normalizing map,
Lemma \ref{lem-iso-esti} (i) and the convexity guarantee the uniform boundedness of the first derivative of $\tilde{r}$ in the rescaled version of \eqref{r_evolution}. Then the regularity theory of uniformly parabolic equation provide the boundedness of the higher derivatives of $\tilde{r}$.  Thus, recalling Lemma \ref{highA}, each time slice $\tilde{X}(\cdot, \tau_{k})$ has a $C^{\infty}$-convergent subsequence to a smooth  strictly convex limit hypersurface $\tilde{M}^*$.  In dimension two, the limit hypersurface $\hat{M}^*$ of the volume preserving anisotropic scalar curvature flow satisfies the equation $\hat{S}^* = C \hat{\psi} (\hat{R}^*)^{1/2}$ for some $C>0$ by Lemma \ref{monotone}.

Suppose that $h_{ij} \ge \epsilon(H+c)g_{ij}$ initially.   We follow the argument in Sect.7 in \cite{C2}.
Since there is a point $p_0$ in $\tilde{M}^*$ with $\tilde{K}(p_0)>0$,
there is an open neighborhood $\tilde{U}$ containing $p_0$ with $\tilde{K}>0$ in $\tilde{U}$.
However the unnormalized $H$ blows up in the open neighborhood $U$ corresponding to $\tilde{U}$ and then
from Theorem \ref{L_infinity_bound} and the scale invariance of $f$, we have $f=0$ in $U$ which implies that
$\tilde{U}$ is totally umbilical. Thus $\tilde{K}$ is constant in $\tilde{M}^*$ so that $\tilde{M}^*$ is a round sphere $S^n$.
\qed

\end{document}